%% file: Preprint_Collocation_Lognormal.tex
\documentclass[11pt]{article}

\usepackage{amsmath,amsfonts,amssymb,amsthm,graphicx,color,mathtools}
\usepackage[normalem]{ulem} 

\usepackage{times}
\usepackage{graphicx}
\usepackage[utf8]{inputenc}
\usepackage{paralist}
\usepackage[colorlinks,pdfdisplaydoctitle]{hyperref}

\include{macros}

\makeatletter
\newcommand{\subjclass}[2][1991]{%
  \let\@oldtitle\@title%
  \gdef\@title{\@oldtitle\footnotetext{#1 \emph{Mathematics subject classification.} #2}}%
}
\newcommand{\keywords}[1]{%
  \let\@@oldtitle\@title%
  \gdef\@title{\@@oldtitle\footnotetext{\emph{Key words and phrases.} #1.}}%
}
\makeatother

\begin{document}
\include{Preprint_Collocation_Lognormal_v2}

\end{document}

%% file: macros.tex
\DeclareMathAlphabet{\mathbf}{OT1}{cmr}{bx}{it}

\newcommand{\supp}{\mathrm{supp}\,}

\newcommand{\mc}{\mathcal}

\DeclareMathAlphabet{\mathbf}{OT1}{cmr}{bx}{it}

\DeclareMathOperator{\essinf}{essinf}
\DeclareMathOperator*{\argmax}{argmax}

\newcommand{\ve}{\mathbf e}

\newcommand{\vi}{\mathbf i}

\newcommand{\vk}{\mathbf k}

\newcommand{\valpha}{{\boldsymbol \alpha}}

\newcommand{\vnu}{{\boldsymbol \nu}}
\newcommand{\vtau}{{\boldsymbol \tau}}

\newcommand{\vxi}{\boldsymbol{\xi}}

\newcommand{\bbN}{\mathbb{N}}
\newcommand{\bbR}{\mathbb{R}}

\newcommand{\bbP}{\mathbb{P}}

\newcommand{\Span}{\mathrm{span}}

\DeclareMathOperator{\e}{\mathrm e}
\newcommand{\ev}[1]{ {\boldsymbol{\mathsf  E}} \left[  #1 \right]}

\renewcommand{\d}{\mathrm{d}}

 \theoremstyle{definition}
 \newtheorem{defi}{Definition}
 
 \newtheorem{rem}[defi]{Remark}
 \newtheorem{propo}[defi]{Proposition}
 \newtheorem{theo}[defi]{Theorem}
 \newtheorem{lem}[defi]{Lemma}
 
 \newtheorem{assum}{Assumption}
 

%% file: Preprint_Collocation_Lognormal_v2.tex
%
%

\title{Convergence of Sparse Collocation for Functions of Countably Many Gaussian Random Variables (with Application to Elliptic PDEs)}
\author{Oliver G. Ernst$^\dag$, Bj\"orn Sprungk$^\dag$ and Lorenzo Tamellini$^\ast$\\
\small
$^\dag$ Department of Mathematics, TU Chemnitz, Germany\\
\small
$^\ast$ Istituto di Matematica Applicata e Tecnologie Informatiche ``E. Magenes" del CNR, Pavia, Italy
}
\date{}
\maketitle

\pagestyle{myheadings}
\thispagestyle{plain}
\markboth{O. G. Ernst, B. Sprungk, L. Tamellini}{Sparse Collocation for Lognormal Diffusion}
\maketitle

%
%

\begin{abstract}
We give a convergence proof for the approximation by sparse collocation of Hilbert-space-valued functions depending on countably many Gaussian random variables. 
Such functions appear as solutions of elliptic PDEs with lognormal diffusion coefficients.
We outline a general $L^2$-convergence theory based on previous work by Bachmayr et al. \cite{BachmayrEtAl2015} and Chen \cite{Chen2016} and establish an algebraic convergence rate for sufficiently smooth functions assuming a mild growth bound for the univariate hierarchical surpluses of the interpolation scheme applied to Hermite polynomials.
We verify specifically for Gauss-Hermite nodes that this assumption holds and also show algebraic convergence w.r.t.\ the resulting number of sparse grid points for this case.
Numerical experiments illustrate the dimension-independent convergence rate.\\[0.5em]
\end{abstract}


%
%
%
%
\section{Introduction}
The elliptic diffusion problem 
\begin{equation} \label{equ:PDE}
	-\nabla \cdot (a(\omega)\,\nabla u(\omega))
	= f
	\quad \text{in } D \subset \mathbb R^d, 
	\qquad
	u(\omega) = 0\;\text{on }\partial D,
	\qquad
	\bbP\text{-a.s.~,}
\end{equation}
with a random diffusion coefficient $a:\Omega\to L^\infty(D)$ with respect to an underlying probability space  $(\Omega, \mathcal A, \bbP)$ has become the standard model problem for numerical methods for solving random PDEs.
For modeling reasons the diffusion field is often taken to have a lognormal probability law, 
which complicates both the study of the well-posedness of the problem 
\cite{Charrier2012,gittelson:logn,sarkis:lognormal,MuglerStarkloff2013} 
as well as the analysis of approximation methods.
One of the challenges is that the most common parametrization of a Gaussian random field 
-- the \emph{Karhunen-Lo\`eve expansion} \cite{Adler1981,GhanemSpanos1991} -- involves a countable number of standard normal random variables
\begin{equation} \label{equ:KLE}
	\log a(x, \omega) = \phi_0(x) + \sum_{m=1}^\infty \phi_m(x) \, \xi_m(\omega), 
\end{equation}
where $\phi_0, \phi_m \in L^\infty(D)$ and $\xi_m\sim N(0,1)$  i.i.d. for $m\in\bbN$, leading to an elliptic PDE with a countably infinite number of random parameters $\vxi = (\xi_m)_{m\in\bbN}\in\bbR^\bbN$.

Besides the stochastic Galerkin method \cite{GhanemSpanos1991,lemaitre:book} 
the most common methods for approximating the solution $u(\vxi)$ of such random or parametric elliptic PDEs are \emph{polynomial collocation methods}.
Early works on such methods for random PDEs considered a finite (if large) 
number of random parameters, a setting also referred to as \emph{finite-dimensional noise} 
\cite{XiuHesthaven2005,BabuskaEtAl2010,NobileEtAl2008a,NobileEtAl2008b}.
In this case the parametric representation of $\log a$ is 
typically obtained by truncating a series expansion of the random field such as \eqref{equ:KLE}. 

The analysis of the problem involving an infinite number of random variables was 
first discussed by Cohen, DeVore and Schwab in 
\cite{CohenDeVoreSchwab2010,CohenDeVoreSchwab2011}
in the simpler setting in which the diffusion field $a$, rather than its logarithm as in \eqref{equ:KLE}, is expanded in a series.
This results in an affine dependence of $a$ on the random variables $\xi_m$, which are, moreover, assumed to have bounded support.  
In this framework the convergence of the best $N$-term approximation of the solution of the diffusion equation by Taylor and Legendre series was shown to be independent of the number of random variables; this result was further refined in the recent paper \cite{BachmayrEtAl2016}.
Employing the theoretical concepts stated in \cite{CohenDeVoreSchwab2010,CohenDeVoreSchwab2011},
Chkifa, Cohen and Schwab analyze in \cite{ChkifaEtAl2014}
collocation methods based on Lagrange interpolation with Leja points
for problems with diffusion coefficients depending linearly on 
an infinite number of bounded random variables, which are adaptive
in the polynomial degree as well as the number of active dimensions or
random variables, respectively.  The adaptive algorithm itself is
related to the earlier work \cite{GerstnerGriebel2003}.  Each
interpolatory approximation gives rise to a quadrature scheme, and
in \cite{SchillingsSchwab2013} Schillings and Schwab consider sparse
adaptive quadrature schemes  in the same setting of \cite{ChkifaEtAl2014} in connection with approximating
expectations with respect to posterior measures in Bayesian inference.
Extensions to the case where the diffusion coefficient $a$ depends 
non-linearly on an infinite number of random variables
with bounded support was discussed in \cite{ChkifaEtAl2015}.

Returning to the original lognormal diffusion problem, i.e.,
with $a$ expanded as in \eqref{equ:KLE} and depending on random variables with unbounded support, 
Hoang and Schwab \cite{HoangSchwab2014} have obtained convergence results on best $N$-term approximation by Hermite polynomials. 
These were recently extended by Bachmayr et al.\ \cite{BachmayrEtAl2015} using a different analytical approach employing 
a weighted $\ell^2$-summability of the coefficients of the Hermite expansion of the solution 
and their relation to partial derivatives.
The theoretical tools provided in \cite{BachmayrEtAl2015} enabled a convergence analysis for adaptive sparse quadrature \cite{Chen2016} 
employing, e.g., Gauss-Hermite nodes for Banach space-valued functions of countably many Gaussian random variables.

In this paper we address the convergence of sparse polynomial collocation for functions of infinitely many Gaussian random variables, such as the solution to the lognormal diffusion problem \eqref{equ:PDE}.
Specifically, we follow the approach of \cite{BachmayrEtAl2015} and \cite{Chen2016} to prove an algebraic convergence rate with respect to the number of grid points for sparse collocation based on Gauss-Hermite interpolation nodes in the case of countably many variables.
In particular, the result applies to solutions $u$ of \eqref{equ:PDE} where $a$ is a lognormal random field.
In addition, we highlight the common ideas surrounding sparse collocation found in the works mentioned above.
The convergence result in terms of the number of collocation points is obtained in two steps: we first link the error to the size of the multi-index set definining the sparse collocation and then derive a bound on the number of points in the associated sparse grid. 
This procedure has been followed also in all the above-mentioned work analyzing the convergence 
of sparse grid quadrature and collocation schemes.
An alternative strategy which instead links the error directly to the number of collocation points by introducing the so-called ``profits'' of each component of the sparse grids, has been discussed in \cite{nobile.eal:optimal-sparse-grids,hajiali.eal:MISC2},
albeit only in the case of random variables with bounded support.

We remark that, besides the classical node families such as Gauss-Hermite and Genz-Keister \cite{GenzKeister1996} for quadrature and interpolation on  $\bbR$ with respect to  a Gaussian measure, Jakeman and Narayan \cite{NarayanJakeman2014} 
have introduced \emph{weighted Leja points}---a generalization of the classical Leja point construction (see e.g.\ \cite{Leja1950,DeMarchi2004} and references therein) to unbounded domains and arbitrary weight functions.
Moreover, they have proved that these node sets possess the correct asymptotic 
distribution of interpolation nodes and illustrate their computational potential in numerical experiments.
Note that such weighted Leja points provide a nested and linearly growing sequence of interpolation nodes.
The analysis of sparse collocation based on normal Leja points, i.e., 
weighted Leja points for a Gaussian measure, is an interesting topic for future research.

The remainder of the paper is organized as follows.
In the next section we introduce the general setting and notation and construct 
the sparse grid collocation operator based on univariate Lagrange interpolation operators.
Section \ref{sec:conv} is devoted to the convergence analysis of such operators.
First, we outline in Subsection \ref{subsec:convGen} the general approaches 
to prove algebraic convergence rates as they can be found in the works mentioned above.
Later, we follow in Subsection \ref{subsec:convSG} the approach of \cite{BachmayrEtAl2015, Chen2016} 
and derive sufficient conditions for the underlying univariate interpolation nodes in order to obtain such rates when approximating 
``countably-variate'' functions of certain smoothness.
Finally, in Subsection \ref{subsec:convGH} we verify these conditions for Gauss-Hermite nodes, 
provide bounds for the number of nodes in the resulting sparse grids, and state a convergence result with respect to this number.
Section \ref{sec:PDE} comes back to our motivation and comments on the application to random elliptic PDEs 
before we verify our theoretical findings in Section \ref{sec:numerics} for a simple boundary value problem in one spatial dimension. 
We draw final conclusions in Section \ref{sec:concl}. 

%
%
\section{Setting and Sparse Collocation}

We consider functions $f$ defined on a parameter domain $\Gamma \subseteq \bbR^\bbN$ taking values in a separable real Hilbert space $\mc H$ with inner product $(\cdot,\cdot)_{\mc H}$ and norm $\|\cdot\|_{\mc H}$.
As our interest lies in the approximation of the dependence of $f:\Gamma \to \mc H$ on $\vxi\in\Gamma$ by multivariate polynomials based on Lagrange interpolation, a minimal requirement is that point evaluation of $f$ at any $\vxi \in \Gamma$ be well-defined. 
Stronger smoothness requirements on $f$ become necessary when deriving convergence rate estimates for the approximations.

We introduce a probability measure $\mu$ on the measurable space $(\bbR^\bbN, \otimes_{m\geq1}\mc B(\bbR))$ as the countable product measure of standard Gaussian measures on $\bbR$, i.e.,
\begin{equation}\label{equ:mu}
	\mu = \bigotimes_{m\geq 1} N(0,1).
\end{equation}
and denote by $L^2_\mu(\Gamma; \mc H)$ the space of all (equivalence classes of) functions with finite second moments with respect to $\mu$ in the sense that
\[
   \int_{\bbR^\bbN} \|f(\vxi)\|_{\mc H}^2 \, \mu(\d\vxi) < \infty
\]
which forms a Hilbert space with inner product
\[
    (f,g)_{L^2_\mu} = \int_{\bbR^\bbN} (f(\vxi), g(\vxi))_{\mc H} \, \mu(\d\vxi).
\]
In the following we require 

\begin{assum}\label{assum:f}
Let $f:\Gamma\to\mc H$ where $\mu(\Gamma) = 1$.
There holds (for a measurable extension of $f$ to $\bbR^\bbN$) that $f\in L^2_\mu(\bbR^\bbN; \mc H)$.
\end{assum}

It is shown, e.g.,\ in \cite[Theorem~2.5]{SchwabGittelson2011}, that the countable tensor product of Hermite polynomials forms an orthonormal basis of $L^2_\mu(\bbR^\bbN; \mc H)$.
Under Assumption~\ref{assum:f} we therefore have
\begin{equation} \label{equ:f_PCE}
	f(\vxi) = \sum_{\vnu \in \mc F} f_{\vnu} \, H_{\vnu}(\vxi), 
	\qquad 
	f_{\vnu} := \int_{\bbR^\bbN} f(\vxi) \, H_{\vnu}(\vxi)\, \mu(\d\vxi) \in \mc H,
\end{equation}
where $H_{\vnu}(\vxi) = \prod_{m\geq1} H_{\nu_m}(\xi_m)$ and $H_\nu$ denotes the univariate Hermite orthonormal polynomial of degree $\nu$ as well as
\[
	\mc F := \left\{\vnu \in \bbN_0^\bbN: |\vnu|_0 < \infty \right\}, 
	\qquad 
	|\vnu|_0 := |\{j \in \bbN: \nu_j > 0 \}|.
\]

%
%
\subsection{Sparse Polynomial Collocation}

The construction of sparse collocation operators below is based on sequences of univariate Lagrange interpolation operators 
$U_k$ mapping into the set $\mc P_k$ of univariate polynomials of degree at most $k\in\bbN_0$.
Thus, 
\[
	(U_k f)(\xi) = \sum_{i=0}^k f(\xi_{i}^{(k)}) \, L_i^{(k)}(\xi), 
	\qquad 
	f\colon \bbR \to \bbR,
\]
where $\{L_i^{(k)}\}_{i=0}^k$ denote the Lagrange fundamental polynomials of degree $k$ associated with the set of $k+1$ distinct interpolation nodes $\Xi^{(k)} := \{ \xi_{0}^{(k)},\xi_{1}^{(k)},\dots,\xi_{k}^{(k)}\}$.

\begin{rem}
It may also be of interest to consider sequences of interpolation operators $U_k$ with a more general degree of polynomial exactness $n(k)$ where $n:\bbN_0\to\bbN_0$ is indecreasing and $n(0) = 0$, see for instance \cite{XiuHesthaven2005,BabuskaEtAl2010,NobileEtAl2008a,NobileEtAl2008b,NobileEtAl2016,nobile.eal:optimal-sparse-grids}.
However, we restrict ourselves to $n(k)=k$ for simplicity.
\end{rem}

We also introduce the \emph{detail operators} 
\[
	\Delta_k := U_k - U_{k-1}, \quad k\geq 0,
\]
where we set $U_{-1} :\equiv 0$, and observe that 
\[
	U_k = U_{k-1} + \Delta_k = \Delta_0 + \Delta_1 + \cdots + \Delta_k~.
\]

\paragraph{Tensorization} 
For any multi-index $\vk = (k_m)_{m \in \mathbb N} \in \mc F$ the \emph{(full) tensor product interpolation operator}
$
   U_\vk := \bigotimes_{m \in \bbN} U_{k_m}
$
is defined by
\begin{equation} \label{eq:tensor-int}
   (U_\vk f)(\vxi)
   =
   \left(\bigotimes_{m \in \bbN} U_{k_m} f \right)(\vxi)
   =
   \sum_{\vi \le \vk} f(\vxi_{\vi}^{(\vk)}) L_\vi^{(\vk)}(\vxi),
  \qquad
  f:\bbR^\bbN \to \bbR,
\end{equation}
where $\vxi_\vi^{(\vk)} \in \bbR^\bbN$ ranges over all points in the Cartesian product
\begin{align}\label{equ:Xi}
    \Xi^{(\vk)} := \bigtimes_{m \in \bbN} \Xi^{(k_m)},
    \quad \text{ with } \quad
    |\Xi^{(\vk)}| = \prod_{m \in \bbN} (1+k_m),
\end{align}
and where
\begin{align}\label{equ:L_nu}
   L_\vi^{(\vk)}(\vxi) := \prod_{m \in \mathbb N} L_{i_m}^{(k_m)}(\xi_m)
\end{align}
is a multivariate polynomial of (total) degree $|\vk|_1 = \sum_m k_m$. 
Note that $L_0^{(0)}(\xi) \equiv 1$; in particular, since $\vk \in \mc F$ 
all but a finite number of factors in \eqref{equ:Xi} and \eqref{equ:L_nu} are equal to one so that the corresponding products can be regarded as finite.
The tensor product interpolation operator $U_\vk$ maps into the multivariate (tensor product) polynomial space
\begin{equation} \label{eq:P_i}
   \mc Q_\vk := \Span \{ \vxi^\vi : 0 \le i_m \le k_m, m \in \bbN \},
   \qquad \vk \in \mc F.
\end{equation}
Note that, since both the univariate polynomial sets of Lagrange 
fundamental polynomials $\{L_i^{(k)}\}_{i=0}^k$ and the
Hermite orthonormal polynomials $\{H_i\}_{i=0}^k$  
form a basis of $\mc P_k$, equivalent characterizations are 

\begin{equation*}
\begin{split}
   \mc Q_\vk
   &=
   \Span \{ L_\vi^{(\vk)} : 0 \le i_m \le k_m, m \in \bbN \} \\
   &=
   \Span \{ H_\vi         : 0 \le i_m \le k_m, m \in \bbN \},
   \qquad \vk \in \mc F.
\end{split}     
\end{equation*}

In order for the tensor product interpolation operator $U_\vk$ to be applicable also to functions defined only on a subset $\Gamma \subset \bbR^\bbN$, we assume the interpolation nodes to all lie in $\Gamma$: 

\begin{assum}\label{assum:Xi}
Let $\Gamma\subset \bbR^\bbN$ denote the domain from Assumption \ref{assum:f}.
For all $\vk \in \mc F$ the Cartesian products of nodal sets $\Xi^{(\vk)}$ given in \eqref{equ:Xi} satisfy
$
	\Xi^{(\vk)} \subset \Gamma.
$
\end{assum}

In the following we denote by $\bbR^\Gamma$ the set of all mappings from $\Gamma$ to $\bbR$.
In analogy to \eqref{eq:tensor-int} we define for any multi-index $\vk \in \mc F$ the tensorized detail operator
\[
	\Delta_{\vk} := \bigotimes_{m \in \bbN} \Delta_{k_m} \colon \bbR^\Gamma \to \mc Q_{\vk}.
\]
Finally, we associate with a finite subset $\Lambda \subset \mc F$ the multivariate polynomial space
\begin{equation} \label{eq:P_Lambda}
   \mc P_{\Lambda} := \sum_{\vi \in \Lambda} \mc Q_\vi
\end{equation}
and define the associated \emph{sparse (polynomial) collocation operator} $U_\Lambda \colon \bbR^\Gamma \to \mc P_{\Lambda}$ by
\begin{equation} 
	U_{\Lambda} := \sum_{\vi \in \Lambda} \Delta_{\vi}.
\end{equation}
We will see that $U_\Lambda$ is exact on $\mc P_\Lambda$ under some natural assumptions on the multi-index set $\Lambda$, for which we first recall some basic definitions given in \cite{CohenDeVore2015,ChkifaEtAl2014,ChkifaEtAl2015}.

\paragraph{Partial orderings and monotone sets of multi-indices} We define a partial ordering on $\mc F$ by 
\begin{align*}
	\tilde \vnu \leq \vnu 
	\quad & :\Leftrightarrow \quad 
	\tilde \nu_m \leq \nu_m \; \forall m\in\bbN
\intertext{as well as}
	\tilde \vnu < \vnu 
	\quad & :\Leftrightarrow \quad 
	\tilde \vnu \leq \vnu \, 
	\text{ and } \tilde \nu_m < \nu_m \text{ for at least one } m\in\bbN
\intertext{and introduce the relation}
	\tilde \vnu \not \leq \vnu 
	\quad & :\Leftrightarrow \quad 
	\tilde\nu_m > \nu_m \text{ for at least one } m\in\bbN.
\end{align*}
We shall call a set of multi-indices $\Lambda \subset \mc F$ \emph{monotone} if $\vnu \in \Lambda$ and $\tilde \vnu \leq \vnu$ together imply that also $\tilde\vnu \in \Lambda$.
Finally, for a multi-index $\vnu \in \mc F$ we define its \emph{rectangular envelope} $\mc R_\vnu$ by
\begin{align*}\label{equ:R_nu}
	\mc R_\vnu & :=  \{\tilde \vnu \in \mc F: \tilde \vnu \leq \vnu \}.
\end{align*}
Note that $\mc R_{\vnu}$ for $\vnu \in \mc F$ is a finite (and monotone) set with cardinality
\begin{equation}\label{eq:card_of_R_nu}
	|\mc R_{\vnu}| = \prod_{m \in \bbN} (1 + \nu_m) < \infty.  
\end{equation}

%
%
\subsection{Polynomial Exactness of Sparse Collocation}

The introduction of the rectangular envelope $\mc R_\vnu $ of a multi-index $\vnu \in \mc F$ permits a convenient characterization of monotone multi-index sets $\Lambda$ and the associated polynomial space $\mc P_\Lambda$ introduced in \eqref{eq:P_Lambda}.

\begin{propo}\label{propo:montone_set}
If $\Lambda \subset \mc F$ is monotone, then
\[
   \Lambda = \bigcup_{\vnu \in \Lambda} \mc R_\vnu
   \quad \text{ and } \quad
   \mc P_\Lambda
   =
   \Span \{ \vxi^\vnu : \vnu \in \Lambda \} 
   =
   \Span \{ H_\vnu : \vnu \in \Lambda \}.
\]
\end{propo}
\begin{proof}
Since $\vnu \in \mc R_\vnu$ for all $\vnu \in \Lambda$ the set on the left is obviously a subset of that on the right.
Conversely, given $\vi \in \mc R_\vnu$ for some $\vnu \in \Lambda$, the definition of $\mc R_\vnu$ implies $\vi \le \vnu$, which in turn implies $\vi \in \Lambda$ by the monotonicity of $\Lambda$.
Moreover, monontonicity also implies
\begin{align*}
	\mc P_\Lambda
	& = \sum_{\vk \in \Lambda} \mc Q_\vk
	=  \Span \{ \vxi^\vi : \vi \leq \vk, \vk \in \Lambda \} 
	= \Span \{ \vxi^\vi : \vi \in \Lambda \}
	= \Span \{ H_\vnu : \vnu \in \Lambda \},
\end{align*} 
where monontonicity is required for the two last equalities.
\end{proof}
In view of Proposition~\ref{propo:montone_set}, $\mc P_\Lambda$ for a multi-index set $\Lambda \subset \mc R_\vk$ represents a sparsification of $\mc Q_\vk$.
In particular, the full tensor product polynomial space $\mc Q_\vk$ coincides with $\mc P_\Lambda$ for $\Lambda = \mc R_\vk$.
Similarly, the full tensor approximation operator $U_\vk$ defined in \eqref{eq:tensor-int} can be expressed as $U_\vk = \sum_{\vi \in \mc R_\vk} \Delta_\vi$.

\begin{propo}
\label{propo:U_Lambda}
Let $\Lambda \subset \mc F$ be a finite and monotone set. Then $U_\Lambda p = p$ for all $p\in \mc P_\Lambda$.
In particular, for all $p\in\mc P_\Lambda$ we have $\Delta_\vi p = 0$ for $\vi \notin \Lambda$.
\end{propo}
\begin{proof}
Observe first that, for any $\vnu, \vi \in \mc F$ such that $\vi \not\leq \vnu$ we have
\[
	\Delta_\vi \vxi^\vnu 
	= 
	\prod_{m \in \bbN} \Delta_{i_m} \xi_m^{\nu_m} 
	=  
	\prod_{m \in \bbN} 
	\underbrace{(U_{i_m} - U_{i_m-1})\xi_m^{\nu_m}}_{= \xi_m^{\nu_m} - \xi_m^{\nu_m} \equiv 0  \text{ for at least one } m} 
	= 0.
\]
It suffices to prove the assertions for all monomials $\vxi^\vnu$ in $\mc P_\Lambda$.
For $\vnu \in \Lambda$ any $\vi \in \mc F \setminus \Lambda$ must satisfy $\vi\not\leq\vnu$ 
and therefore $\Delta_\vi \vxi^\vnu = 0$, proving the second assertion.
We conclude that
\begin{align*}
	U_\Lambda \vxi^\vnu
	& = \sum_{\vi \in \Lambda} \Delta_{\vi} \vxi^\vnu
	= \sum_{\vi \in \Lambda \cap \mc R_\vnu} \Delta_{\vi} \vxi^\vnu
	= \sum_{\vi \in \mc R_\vnu} \Delta_{\vi} \vxi^\vnu,
\end{align*}
where the third equality follows from the fact that $\mc R_\vnu \subseteq \Lambda$ for all $\vnu \in \Lambda$ due to the monotonicity of $\Lambda$. 
The proof concludes with
\begin{align*}
	U_\Lambda \vxi^\vnu 
	& = \sum_{\vi \in \mc R_\vnu} \Delta_{\vi} \vxi^\vnu 
	= \sum_{\vi \in \mc R_\vnu} \left(\prod_{m \in \bbN} \Delta_{i_m} \xi_m^{\nu_m} \right)
	= \prod_{m \in \bbN} \left(\sum_{i_m=0}^{\nu_m} \Delta_{i_m} \xi_m^{\nu_m}\right) 
	= \prod_{m \in \bbN} U_{\nu_m} \xi_m^{\nu_m}\\
	& = \prod_{m \in \bbN} \xi_m^{\nu_m}
	= \vxi^\vnu.
\end{align*}
Note that the third equality is obtained by rewriting a (finite) product of sums: since $\vnu \in \mc F$ there exists an $M \in \bbN$ such that $\nu_m = 0$ for $m>M$.
For such $m$ we have $\Delta_{i_m}^{\nu_m} \xi_m^{\nu_m} = \Delta_0 \xi_m^0 \equiv 1$ and therefore
\begin{align*}
	\prod_{m \in \bbN} \left( \sum_{i_m=0}^{\nu_m} \Delta_{i_m} \xi^{\nu_m}_m \right)
	&= 
	\left(\Delta_0 \xi^{\nu_1}_1 + \cdots + \Delta_{\nu_1} \xi^{\nu_1}_1 \right) 
	\cdots 
	\left(\Delta_0\xi^{\nu_M}_M + \cdots + \Delta_{\nu_M}\xi^{\nu_M}_M \right)\\
	&= 
	\sum_{\substack{\vi \in \bbN_0^M \\ i_m \leq \nu_m}} 
	\Delta_{i_1}\xi^{\nu_1}_1 \cdots  \Delta_{i_M} \xi^{\nu_M}_M
	= 
	\sum_{\vi \in \mc R_\vnu} \left(\prod_{m \in \bbN} \Delta_{i_m} \xi^{\nu_m}_m\right).
\end{align*}
\end{proof}

Proposition \ref{propo:U_Lambda} can be seen as an extension of 
\cite[Proposition 1]{back.nobile.eal:comparison} 
to general monotone multi-index sets as well as an extension of \cite[Theorem 6.1]{CohenDeVore2015} and \cite[Theorem 2.1]{ChkifaEtAl2014} to interpolation operators $U_i$ with non-nested node sets.
As mentioned in \cite[p. 89]{CohenDeVore2015}, if the set $\Lambda$ is not monotone then $U_\Lambda$ will not be exact on $\mc P_\Lambda$ in general.
However, the exactness on $\mc P_\Lambda$ is a crucial property in the subsequent convergence analysis and we therefore choose to work exclusively with monotone sets $\Lambda$.

%
%
\subsection{Sparse Grid Associated with $U_\Lambda$}

The construction of $U_\Lambda f$ for $f\colon \Gamma \to \bbR$ consists of a linear combination of tensor product interpolation operators requiring the evaluation of $f$ at certain multivariate nodes.
We shall refer to the collection of these nodes as the \emph{sparse grid} $\Xi_\Lambda \subset \Gamma$ associated with $\Lambda$.
For a monotone and finite set $\Lambda\subset \mc F$ there holds
\begin{equation}\label{equ:sparse_grid}
	\Xi_\Lambda = \bigcup_{\vi \in \Lambda} \Xi^{(\vi)},
\end{equation}
because for $\vi\in\mc F$ we have
\[
	\Delta_\vi f 
	= \Big[ \bigotimes_{m\geq 1} (U_{i_m} - U_{i_m-1})  \Big]f
	= \sum_{\vi - \boldsymbol 1 \leq \vk \leq \vi} (-1)^{|\vi-\vk|_1} \Big[\bigotimes_{m\geq 1} U_{k_m}\Big]f,
\]
i.e., for computing $\Delta_\vi f$ we need to evaluate $f$ at
\[
	\Xi^{(\vi),\Delta} := \bigcup_{\vi - \boldsymbol 1 \leq \vk \leq \vi} \Xi^{(\vk)}.
\] 
Since $\Lambda$ is a monotone set, the resulting sparse grid for $U_\Lambda = \sum_{\vi \in \Lambda} \Delta_\vi$ is 
\begin{align*}
    \Xi_\Lambda
    =
	\bigcup_{\vi \in \Lambda} \Xi^{(\vi),\Delta}
	& = 
	\bigcup_{\vi \in \Lambda}
	\,
	\bigcup_{\vi - \boldsymbol 1 \leq \vk \leq \vi} \Xi^{(\vk)}
	= \bigcup_{\vi \in \Lambda} \Xi^{(\vi)}.
\end{align*}
We remark that the unisolvence on $\mc P_\Lambda$ of point evaluation on $\Xi_\Lambda$ is discussed in \cite[Theorem 6.1]{CohenDeVore2015}.

%
%
\section{Convergence Analysis}\label{sec:conv}

In this section we analyze the error 
\[
	\|f - U_\Lambda f\|_{L^2_\mu}, \qquad f\colon \Gamma \to \mc H, 
\]
where $\|\cdot\|_{L^2_\mu}$ denotes the norm in $L^2_\mu(\bbR^\bbN; \mc H)$, $f$ is assumed to satisfy Assumption \ref{assum:f} and $\Lambda\subset \mc F$ is required to be monotone and finite.
Our first goal here is to establish a convergence rate $s>0$ for the error of $U_{\Lambda_N}$ 
for a nested sequence $\Lambda_N$ of monotone subsets of $\mc F$ with $|\Lambda_N| = N$, i.e.,
\begin{align}\label{equ:conv}
	\|f - U_{\Lambda_N} f\|_{L^2_\mu} \leq C N^{-s}, \qquad f\colon \Gamma \to \mc H,
\end{align}
where $C<\infty$ may depend on $f$ as well as the univariate nodal sets.
The line of proof we present here follows and builds upon the works \cite{Chen2016,HoangSchwab2014,BachmayrEtAl2015}.
We complement this convergence rate with a bound on the number of collocation points associated with a given multi-index set.

%
%

\subsection{General Convergence Results}\label{subsec:convGen}

The subsequent error analysis for the sparse collocation operator $U_\Lambda$ is based on the representation of multivariate functions $f \in L^2_\mu(\bbR^\bbN; \mc H)$ in the orthonormal basis of multivariate Hermite polynomials $H_\vnu$. 
We shall therefore examine the worst-case approximation error of any $U_\Lambda$ applied to a given multivariate Hermite basis polynomial $H_\vnu$.
To this end we define 
\begin{equation} \label{equ:c_nu}
	c_\vnu := \sup_{\Lambda \subset \mc F, |\Lambda|<\infty} \|(I - U_{\Lambda}) H_{\vnu}\|_{L^2_\mu},
	\qquad 
	\vnu \in \mc F.
\end{equation}
This quantity is finite since $\Delta_\vi H_\vnu = 0$ for $\vi \not \leq \vnu$ and hence
\[
	c_\vnu
	= 
	\max_{\Lambda \subseteq \mc R_\vnu} \|(I - U_{\Lambda}) H_{\vnu}\|_{L^2_\mu},
\]
where the maximum is taken over a finite set.
The quantities $c_\vnu$ also measure the deviation of the error of oblique projection $U_\Lambda$ from that of orthogonal projection, as these numbers would all be zero or one if $U_\Lambda$ is replaced with the $L^2_\mu$-orthogonal projection onto $\mc P_\Lambda$.
Moreover, we obtain the following bound:
\begin{propo} \label{thm:c_nu}
For all  $\vnu\in \mc F$ the quantity $c_\vnu$ defined in \eqref{equ:c_nu} satisfies
\[
	c_\vnu \leq  \sum_{\vi \in \mc R_\vnu} \|\Delta_\vi H_\vnu\|_{L^2_\mu}.
\]
In particular, if for the univariate Hermite polynomials there exists $\theta \geq 0$ and $K\geq1$ such that
\begin{align}\label{equ:Delta_i_bound}
	\|\Delta_i H_\nu\|_{L^2_\mu} \leq (1+ K\nu)^\theta
	\qquad \text{ for all }i \in\bbN_0,
\end{align}
where we have denoted the univariate Gaussian measure again by $\mu$, then
\begin{align} \label{equ:c_nu_bound}
	c_\vnu 
	& \leq 
	\prod_{m \in \bbN} (1 + K\nu_m)^{\theta+1},
	\qquad \vnu \in \mc F.
\end{align}
\end{propo}
\begin{proof}
In view of Proposition \ref{propo:U_Lambda} we have $H_\vnu = U_{\vnu}H_\vnu = \sum_{\vi\in\mc R_\vnu} \Delta_\vi H_\vnu$ 
and, particularly, $\Delta_\vi H_\vnu = 0$ for $\vi \not \in \mc R_\vnu$, since $H_\vnu \in \mc P_{\mc R_\vnu}$.
Therefore
\begin{align*}
	(I - U_{\Lambda}) H_{\vnu}
	&= 
	\sum_{\vi \in \mc R_\vnu} \Delta_i H_\vnu - \sum_{\vi \in \Lambda} \Delta_i H_\vnu 
	= 
	\sum_{\vi \in \mc R_\vnu} \Delta_i H_\vnu - \sum_{\vi \in \Lambda\cap \mc R_\vnu} \Delta_i H_\vnu\\
	& = 
	\sum_{\vi \in \mc R_\vnu\setminus \Lambda} \Delta_i H_\vnu,
\end{align*}
giving 
\begin{align*}
	c_\vnu 
	&= 
	\max_{\Lambda \subseteq \mc R_\vnu} \|(I - U_{\Lambda}) H_{\vnu}\|_{L^2_\mu}
	\leq 
	\max_{\Lambda \subseteq \mc R_\vnu} \sum_{\vi \in \mc R_\vnu\setminus \Lambda} \|\Delta_i H_\vnu\|_{L^2_\mu}
	\leq 
    \sum_{\vi \in \mc R_\vnu} \|\Delta_i H_\vnu\|_{L^2_\mu}.
\end{align*}
Moreover, if \eqref{equ:Delta_i_bound} holds, then
\begin{align*}
	c_\vnu 
	& \leq \sum_{\vi \in \mc R_\vnu} \|\Delta_\vi H_\vnu\|_{L^2_\mu}
	= 
	\sum_{\vi \in \mc R_\vnu} \prod_{m \in \bbN} \|\Delta_{i_m} H_{\nu_m}\|_{L^2_\mu}
	\leq 
	\sum_{\vi \in \mc R_\vnu} \prod_{m \in \bbN} (1+ K\nu_m)^\theta \\
	&= 
	|\mc R_\vnu| \prod_{m \in \bbN} (1+ K\nu_m)^\theta 
	\leq 
	\prod_{m \in \bbN} (1 + K\nu_m)^{\theta+1}.
\end{align*}
where we have used \eqref{eq:card_of_R_nu} and $K \geq 1$ in the last inequality.
\end{proof}

\begin{rem}\label{rem:c_nu}
Bounds such as \eqref{equ:Delta_i_bound} can often be found in the sparse collocation or sparse quadrature literature, e.g., for quadrature operators applied to Hermite polynomials \cite{Chen2016}, norms of quadrature operators on bounded domains \cite{SchillingsSchwab2013} or Lebesgue constants for Leja points \cite{ChkifaEtAl2015}. Numerical estimates for the specific case of Genz-Keister points
have been provided in \cite{back.nobile.eal:lognormal}. 
\end{rem}

The following lemma provides a natural starting point for bounding the approximation error of $U_\Lambda f$ for monotone subsets $\Lambda$.
The proof follows the same line of argument as the proof of \cite[Lemma 3.2]{Chen2016}. 

\begin{lem}[{cf.\ \cite[Lemma 3.2]{Chen2016}}] \label{lem:I-U}
For a finite and monotone subset $\Lambda \subset \mc F$ there holds
\begin{equation}
	\left\| f - U_{\Lambda} f \right\|_{L^2_\mu} 
	\leq 
	\sum_{\vnu \in \mc F\setminus \Lambda} c_{\vnu} \|f_\vnu\|_{\mc H}.
\end{equation}
\end{lem}
\begin{proof}
Due to the monotonicity of $\Lambda$ we can apply Proposition \ref{propo:U_Lambda} and obtain
\begin{align*}
	\left\| f - U_{\Lambda} f \right\|_{L^2_\mu} 
	&= 
	\biggl\| \sum_{\vnu \in \mc F} f_{\vnu} \, (I - U_{\Lambda}) H_{\vnu} (\vxi) \biggr\|_{L^2_\mu}
	= 
	\biggl\| 
	\sum_{\vnu \in \mc F\setminus \Lambda} f_{\vnu} \, (I - U_{\Lambda}) H_{\vnu} (\vxi) 
	\biggr\|_{L^2_\mu} \\
	&\leq  
	\sum_{\vnu \in \mc F\setminus \Lambda} 
	\|f_{\vnu}\|_{\mc H} \|(I - U_{\Lambda}) H_{\vnu}\|_{L^2_\mu}
	\leq 
	\sum_{\vnu \in \mc F\setminus \Lambda} c_{\vnu} \|f_{\vnu}\|_{\mc H}.
\end{align*}
\end{proof}

Building on Lemma~\ref{lem:I-U} the approximation error $\| f - U_{\Lambda} f \|_{L^2_\mu}$ may be further bounded given summability results for the sequence $(c_\vnu\|f_{\vnu}\|_{\mc H})_{\vnu\in\mc F}$.
The key result here is known as \emph{Stechkin's lemma} which provides a decay rate for the $\ell^q$-tail of an $p$-summable sequence for $q>p$ and is due to Stechkin \cite{Stechkin1955} for $q=2$ (cf.\ also \cite[Lemma 3.6]{CohenDeVore2015}).

\begin{lem}[Stechkin] 
\label{lem:Stechkin}
Let $0<p<q<\infty$ and let 
\[
	(a_\vnu)_{\vnu\in\mc F} \in \ell^p(\mc F)
	:= 
	\left\{ (b_{\vnu})_{\vnu\in\mc F} \colon \sum_{\vnu\in\mc F} |b_\vnu|^p < \infty \right\}
\]
be a sequence of nonnegative numbers.
Then for $\Lambda_N$ denoting the set of multi-indices $\vnu$ corresponding to the $N$ largest elements $a_\vnu$, there holds
\begin{equation} \label{eq:stechkin}
	\Biggl(\sum_{\vnu\notin \Lambda_N} a^q_\vnu\Biggr)^{1/q} 
	\leq 
	\|(a_\vnu)_{\vnu\in\mc F}\|_{\ell^p} (N+1)^{-s}, \qquad s = \frac 1p - \frac 1q~.
\end{equation}
\end{lem}

The index sets $\Lambda_N$ in Stechkin's lemma associated with the $N$ largest sequence elements are not necessarily monotone and, therefore Lemma \ref{lem:I-U} and Lemma \ref{lem:Stechkin} can not be combined to bound the error without additional assumptions.
An obvious way to ensure monotonicity of the sets $\Lambda_N$ in Stechkin's lemma is to assume the sequence $(a_\vnu)$ to be \emph{nonincreasing}, i.e.,
\[
	\vnu \leq \tilde \vnu 
	\quad \Rightarrow \quad 
	a_\vnu \geq a_{\tilde \vnu}.
\]
This leads to 

\begin{theo} \label{theo:conv1}
Let Assumptions \ref{assum:f} and \ref{assum:Xi} be satisfied and let there exist a nonincreasing sequence $(\hat c_{\vnu})_{\vnu \in \mc F}\in \ell^p(\mc F)$ with $p\in(0,1)$ such that
\begin{equation*}
	c_{\vnu} \|f_{\vnu}\|_{\mc H} \leq \hat c_{\vnu} \qquad \forall \vnu \in \mc F.
\end{equation*}
Then there exists a nested sequence $(\Lambda_N)_{N\in\bbN}$ of finite and monotone subsets $\Lambda_N\subset \mc F$ with $|\Lambda_N|=N$ such that \eqref{equ:conv} holds with rate $s = 1/p - 1$.
\end{theo}

We will provide a proof below.
The convergence analysis in \cite{ChkifaEtAl2015,SchillingsSchwab2013} for sparse quadrature and interpolation in case of bounded $\Gamma$ follows Theorem~\ref{theo:conv1}, although sometimes hidden in the details.
There the authors employ explicit bounds on the norms of the Legendre or Taylor coefficients of $f:\Gamma \to \mc H$ to construct a dominating and nonincreasing sequence $(\hat c_{\vnu})_{\vnu \in \mc F}\in \ell^p(\mc F)$, $p\in(0,1)$.

In our setting it is, however, not always possible to derive explicit bounds on the norm of the Hermite coefficients $\|f_\vnu\|_{\mc H}$.
In \cite{BachmayrEtAl2015} a technique was developed which relies on somewhat implicit bounds on $\|f_\vnu\|_{\mc H}$ via a weighted $\ell^2$-summability property.
We adapt this approach to the current setting in

\begin{theo} \label{theo:conv2}
Let Assumptions \ref{assum:f} and \ref{assum:Xi} be satisfied and let there exist a sequence $(b_{\vnu})_{\vnu \in \mc F}$ of positive numbers such that
\begin{equation} \label{equ:sum_b_nu_f_nu}
	\sum_{\vnu \in \mc F} b_{\vnu} \|f_{\vnu}\|^2_{\mc H} < \infty
\end{equation}
as well as another nonincreasing sequence $(\hat c_{\vnu})_{\vnu \in \mc F} \in \ell^p(\mc F)$, $p\in(0,2)$, for which 
\[
	 \frac{c_{\vnu}}{b^{1/2}_\vnu} \leq \hat c_{\vnu} \qquad \forall \vnu\in\mc F.
\]
Then there exists a nested sequence $(\Lambda_N)_{N\in\bbN}$ of finite and monotone subsets $\Lambda_N\subset \mc F$ with $|\Lambda_N|=N$ such that \eqref{equ:conv} holds with rate $s = 1/p - 1/2$.
\end{theo}
\begin{proof}[Proof of Theorem \ref{theo:conv1} and Theorem \ref{theo:conv2}]
Let $\Lambda_N$ be the set of multi-indices $\vnu$ corresponding to the $N$ largest elements of $(\hat c_{\vnu})_{\vnu \in \mc F}$.
Then each $\Lambda_N$ is monotone and the sequence $(\Lambda_N)_{N\in\bbN}$ can be chosen to be nested.

If the assumption of Theorem \ref{theo:conv1} hold, we can apply Lemma \ref{lem:I-U} and Stechkin's lemma with $q=1>p$ to obtain
\begin{align*}
	\left\| f - U_{\Lambda_N} f \right\|_{L^2_\mu} 
	&\
	\leq \sum_{\vnu \in \mc F\setminus \Lambda_N} c_{\vnu} \|f_{\vnu}\|_{\mc H}
	\leq 
	\sum_{\vnu \in \mc F\setminus \Lambda_N} \hat c_{\vnu}
	\leq 
	C(N+1)^{-(1/p-1)}
\end{align*}
where $C = \|(\hat c_\vnu)_{\vnu\in\mc F}\|_{\ell^p}$.

If the assumptions of Theorem \ref{theo:conv2} hold, Lemma~\ref{lem:I-U} combined with the Cauchy-Schwarz inequality and Stechkin's lemma for $q=2>p$ give 
\begin{align*}
	\left\| f - U_{\Lambda_N} f \right\|_{L^2_\mu} 
	&\leq 
	\sum_{\vnu \in \mc F\setminus \Lambda_N} c_{\vnu} \|f_{\vnu}\|_{\mc H}
	= \sum_{\vnu \in \mc F\setminus \Lambda_N} \left( \frac{c_{\vnu}}{b_\vnu^{1/2}}\right) \, \left( b_\vnu^{1/2} \|f_{\vnu}\|_{\mc H} \right)\\
	& 
	\leq 
	\left( \sum_{\vnu \in \mc F\setminus \Lambda_N} b_{\vnu} \|f_{\vnu}\|^2_{\mc H} \right)^{1/2}
	\cdot
	\left( \sum_{\vnu \in \mc F\setminus \Lambda_N} \frac {c^2_{\vnu}}{b_\vnu} \right)^{1/2} \\
	&\leq 
	\left( \sum_{\vnu \in \mc F} b_{\vnu} \|f_{\vnu}\|^2_{\mc H} \right)^{1/2} 
	\cdot 
	\left( \sum_{\vnu \in \mc F\setminus \Lambda_N} \hat c^2_{\vnu}\right)^{1/2} \\
	&\leq 
	C (N+1)^{-(1/p-1/2)},
\end{align*}
where now $C = \| (b^{1/2}_\vnu\|f_\vnu\|)_{\vnu\in\mc F}\|_{\ell^2} \cdot \|(\hat c_\vnu)_{\vnu\in\mc F}\|_{\ell^p}$, respectively.
\end{proof}

\begin{rem}
Another application of the weighted $\ell^2$-summability property \eqref{equ:sum_b_nu_f_nu} is the analysis of sparse quadrature given in \cite{Chen2016}, where the author employs the slightly different estimate
\begin{align*}
	\sum_{\vnu \in \mc F\setminus \Lambda_N} c_{\vnu} \|f_{\vnu}\|_{\mc H}
	& \leq
	\sup_{\vnu\in\mc F \setminus \Lambda_N} b_\vnu^{q - 1/2} 
	\sum_{\vnu \in \mc F\setminus \Lambda_N} 
	\frac{c_{\vnu}}{b_\vnu^{-q}} \, b_\vnu^{1/2} \|f_{\vnu}\|_{\mc H}.
\end{align*}
After showing that the series on the right is bounded and applying Stechkin's lemma to $(b^{q-1/2}_\vnu)_{\vnu\in\mc F}$, this yields the same convergence rate as stated in Theorem \ref{theo:conv2}.
\end{rem}

\begin{rem}
We mention that sparse collocation attains a smaller convergence rate than 
best $N$-term approximation in case the assumptions of Theorem \ref{theo:conv2} hold.
Namely, under these assumptions the best $N$-term rate is $s= \frac 1p$, see \cite[Theorem 1.2]{BachmayrEtAl2015}.
This reduced convergence rate is not caused by the additional factors $c_\vnu$ in the error analysis of sparse collocation.
The reason for the slower rate is missing orthogonality: 
in the proof of Lemma \ref{lem:I-U} we could not apply Parseval's identity and had to use the triangle inequality to bound the error.
This led to bounds in terms of $\|f_\vnu\|_{\mc H}$ rather than $\|f_\vnu\|^2_{\mc H}$ as in the case of orthogonal projections, e.g., best $N$-term approximations.
\end{rem}

We emphasize that the construction of such a nonincreasing, $p$-summable dominating sequence is by no means trivial.
Without the first property we can not conclude that the multi-index sets $\Lambda_N$ occurring in Stechkin's lemma are monotone, which in turn is needed to use Lemma \ref{lem:I-U} as the starting point of our error analysis.
Of course, we could consider monotone envelopes $\Lambda_N\subset \tilde \Lambda_N$ of $\Lambda_N$, but their size can grow quite rapidly with $N$ (e.g., polynomially or even faster, see counterexample below).
Moreover, it is not at all obvious that for a sequence $(a_\vnu)_{\vnu\in\mc F}\in\ell^p(\mc F)$ there exists a dominating and nonincreasing $(\hat a_\vnu)_{\vnu\in\mc F}\in\ell^p(\mc F)$.
In particular, we provide the following counterexample: let $\mc F = \bbN$ and define $a_n$, $n\in\bbN$ by
\[
	a_n = \begin{cases} \frac 1{m^2}, & n = \sum^m_{k=1} k, \\ 0, & \text{otherwise}, \end{cases}
\]
i.e., $a_1 = 1, a_2 = 0, a_3 = \frac 14, a_4 = 0, a_5 = 0, a_6 = \frac 19, a_7 = 0, \ldots, a_9=0, a_{10} = \frac 1{16}, a_{11} = 0, \ldots$~.
Then $(a_n)_{n \in \bbN} \in \ell^1(\bbN)$.
The smallest positive nonincreasing sequence $(\hat a_n)_{n\in\bbN}$ dominating $(a_n)_{n\in\bbN}$ is given by $\hat a_n := \sup_{m \geq n} |a_{m}|$, see \cite[Section 3.8]{CohenDeVore2015}.
In our case, we get 
\[
	\hat a_n = \frac 1{m^2} \qquad \text{for each $n$ such that} \qquad 1+ \sum^{m-1}_{k=1} k \leq n \leq \sum^m_{k=1} k
\]
and, thus, 
\[
	\sum_{n=1}^\infty |\hat a_n| = \sum_{m=1}^\infty m\frac 1{m^2} = \infty.
\]
Although the example is somewhat pathological, it illustrates that for $(a_\vnu)\in\ell^p(\mc F)$  a $p$-summable nonincreasing dominating sequence need not exist.
  
%
%

\subsection{Sufficient Conditions for Weighted Summability and Majorization}\label{subsec:convSG}
We now follow the strategy of Theorem \ref{theo:conv2} and study under which requirements the assumptions of Theorem \ref{theo:conv2} hold.
To this end we recall a result from \cite{BachmayrEtAl2015} for weighted $\ell^2$-summability of Hermite coefficients $\|f_\vnu\|_{\mc H}$ given the following smoothness conditions on $f$:

\begin{assum} \label{assum:df}
Let $f$ satisfy Assumption \ref{assum:f}. 
There exists an integer $r\in\bbN_0$ and a sequence of positive numbers $(\tau^{-1}_m)_{m\in\bbN} \in \ell^p(\bbN)$, $p\in(0,2)$, such that
\begin{enumerate}[(a)]
\item 
for any $\valpha\in\mc F$ with $|\valpha|_\infty \leq r$ the (weak) partial derivative $\partial^\valpha f$ exists and satisfies $\partial^\valpha f \in L^2_\mu(\bbR^\bbN; \mc H)$, \item
there holds
\begin{equation} \label{equ:bound_partial_weighted}
	\sum_{|\valpha|_\infty \leq r} \frac{\vtau^{2\valpha}}{\valpha!} \|\partial^\valpha f\|^2_{L^2_\mu} <\infty,
\end{equation}
where $\vtau^{\valpha} = \prod_{m=1}^\infty \tau_m^{\alpha_m}$ and $\valpha! = \prod_{m=1}^\infty \alpha_m!$.
\end{enumerate}
\end{assum}

Observe that the sum in \eqref{equ:bound_partial_weighted} is actually a series, because $\valpha$ has infinitly many
components and therefore there are countably many vectors such that $|\valpha|_\infty \leq r$.
Assumption~\ref{assum:df}(a) states that we require a \emph{finite} order of partial differentiability of $f$, i.e., up to order $r$ 
with respect to each variable $\xi_m$, and, maybe more importantly, Assumption~\ref{assum:df}(b) 
asks for a \emph{weighted square-summability} of the $L^2_\mu$-norms of the corresponding partial derivatives.
The latter, in particular, implies bounds of the form
\[
	\|\partial^\valpha f\|_{L^2_\mu} \leq K \sqrt{\valpha!}\, \vtau^{-\valpha}, 
	\qquad
	|\valpha|_0 \le r,
\]
since otherwise the summability requirement \eqref{equ:bound_partial_weighted} would not hold.
Recalling that $(\tau^{-1}_m)_{m\in\bbN} \in \ell^p(\bbN)$ this bound implies that, e.g., the $L^2_\mu$-norm of the derivative $\partial^{\alpha}_{\xi_m} f$, $\alpha \leq r$, decays if $m\to\infty$.

The following result shows that the smoothness condition of Assumption~\ref{assum:df} implies the first condition \eqref{equ:sum_b_nu_f_nu} of Theorem~\ref{theo:conv2}:

\begin{theo}[cf. {\cite[Theorem 3.1]{BachmayrEtAl2015}}] \label{theo:Bachmayr1}
Let Assumption \ref{assum:df} be satisfied. 
Then, with the weights
\begin{align} \label{equ:b_nu}
	b_\vnu 
	& = b_\vnu(\vtau,r) 
	= \sum_{|\valpha|_\infty \leq r} \binom{\vnu}{\valpha} \vtau^{2\valpha}
	= \prod_{m\geq1} \left(\sum_{l=0}^r \binom{\nu_m}{l} \tau_m^{2l}\right), \qquad \vnu \in \mc F,
\end{align}
where
\[
	\binom{\vnu}{\valpha} := \prod_{m\geq1} \binom{\nu_m}{\alpha_m}
	\quad \text{ and } \quad
	\binom{\nu_m}{\alpha_m} := 0 \; \text{ if } \; \alpha_m>\nu_m,
\]
there holds
\begin{equation}\label{eq:the_sums_of_b_and_tau}
	\sum_{\vnu \in \mc F} b_\vnu \|f_\vnu\|_{\mc H}^2
	=
	\sum_{|\valpha|_\infty \leq r} \frac{\vtau^{2\valpha}}{\valpha!} \|\partial^\valpha f\|^2_{L^2_\mu}
	< \infty.  
\end{equation}
\end{theo}
(We mention in passing that in \cite{BachmayrEtAl2015} the assertion of Theorem \ref{theo:Bachmayr1} was actually proven without requiring that 
both series in \eqref{eq:the_sums_of_b_and_tau} be finite.) 
To apply Theorem~\ref{theo:conv2} it remains to prove the existence of a nonincreasing and $p$-summable sequence which dominates $c_\vnu / b^{1/2}_\vnu$, $\vnu\in\mc F$.
Since the $b_\vnu$ are explicitly given in \eqref{equ:b_nu}, this boils down to the question, how fast the projection errors $c_\vnu$ are allowed to grow.
As it turns out, a polynomial growth w.r.t. $\vnu$ as given in \eqref{equ:c_nu_bound} in Proposition~\ref{thm:c_nu} is sufficient.
We therefore state the following lemma, which is strongly based on the techniques developed in the proofs of \cite[Lemma 5.1]{BachmayrEtAl2015} and \cite[Lemma 3.4]{Chen2016}.

\begin{lem}
\label{lem:c_nu_b_nu}
Let there exists a $\theta \geq0$ and a $K \geq 1$ such that
\[
	c_\vnu \leq \prod_{m\geq1}^\infty (1+ K \nu_m)^{\theta+1}, \qquad \vnu\in\mc F.
\]
Then for any increasing sequence $(\tau_m)_{m\in\bbN}$ such that $\sum_{m\geq 1} \tau^{-p}_m < \infty$ for a $p > 0$ and for any $r > 2(\theta+1) + \frac 2p$ there exists a nonincreasing sequence $(\hat c_\vnu)_{\vnu\in\mc F} \in \ell^p(\mc F)$ such that
\[
	\frac{c_\vnu}{b^{1/2}_\vnu} \leq \hat c_\vnu \qquad \forall \vnu\in\mc F,
\]
where $b_\vnu = b_\vnu(\tau,r)$ is as in \eqref{equ:b_nu}.
\end{lem}

\begin{proof}
We start with constructing the dominating sequence $(\hat c_\vnu)_{\vnu \in \mc F}$ 
and show afterwards that it belongs to $\ell^p(\mc F)$ and is nonincreasing. 
In the following we use the notation $a \wedge b := \min(a,b)$ and $a \vee b := \max(a,b)$.

\paragraph{Step 1: Constructing $\hat c_\vnu$}
We get due to
\[
	\binom{\nu_m}{\nu_m \wedge r} \tau_m^{2(\nu_m\wedge r)} \leq \binom{\nu_m}{r} \tau_m^{2r} \leq \sum_{l=0}^r \binom{\nu_m}{l} \tau_m^{2l}
\]
that
\begin{align} \label{equ:proof_chen_star}
	\frac {c^2_\vnu}{b_\vnu} 	
	& \leq \prod_{m\geq1} \frac{(1+K\nu_m)^{2(\theta+1)}}{\sum_{l=0}^r \binom{\nu_m}{l} \tau_m^{2l}}
	\leq \prod_{m\geq1} \frac{(1+K\nu_m)^{2\theta+2}}{\binom{\nu_m}{\nu_m \wedge r} \tau_m^{2(\nu_m\wedge r)}}
	= \prod_{m\geq1} \tau_m^{-2(\nu_m\wedge r)} h(\nu_m)
\end{align}
where we defined the auxiliary function $h(n) := \frac{(1+Kn)^{2\theta+2}}{\binom{n}{n \wedge r}}$, $n\in\bbN$.
We will now derive bounds for $h(n)$ as well as for $\tau_m^{-2(\nu_m\wedge r)}$ in order to construct a dominating sequence $\hat c_\vnu$.

For $n\leq r$ we get $h(n) = (1+Kn)^{2\theta+2}$, but for $n>r$ holds
\[
	h(n) = \frac{(1+Kn)^{2\theta+2}}{\binom{n}{r}} = \frac{r! \, (1+Kn)^{2\theta+2}}{(n+1) \cdots(n+r)}.
\]
Thus, we have $h \in \mc O(n^{2\theta+2-r})$, i.e., there exists a $C_h \in[1,\infty)$ such that
\[
	h(n) \leq C_h n^{2\theta+2-r} =: \hat h(n) \qquad \forall n\in\bbN.
\]
By setting $\hat h(0) := 1 = h(0)$, we get $h(n) \leq \hat h(n)$ for all $n\in\bbN_0$.

Furthermore, since $(\tau^{-1}_m)_{m\in\bbN} \in \ell^p(\bbN)$ we have $\tau_m\to\infty$ as $m\to\infty$.
Thus, there exists an $M\in\bbN$ such that $\tau_m \geq \sqrt C_h$ for $m\geq M$ and $\tau_m \leq \sqrt{C_h}$ for $m< M$.
We define 
\[
	\hat \tau_m := \sqrt{C_h} \vee \tau_m, \qquad m\in\bbN,
\]
and notice that $\hat \tau_m \geq 1$ as well as $(\hat \tau^{-1}_m)_{m\in\bbN}\in\ell^p(\bbN)$ by assumption.
Moreover, we obtain for $m\geq M$
\[
	\tau_m^{2(\nu_m\wedge r)} = \hat \tau_m^{2(\nu_m\wedge r)} \geq \hat \tau^{2(\nu_m\wedge 1)}_m, \qquad \forall \nu_m\in\bbN_0,
\]
since $\tau_m = \hat \tau_m \geq \sqrt{C_h} \geq 1$ in this case.
Further, let us define 
\[
	C_\tau := \min_{m\geq1} \min_{n=0,\ldots,r} \frac {\tau^{2n}_m}{C_h^{n\wedge 1}} > 0
\]
which then yields for $1\leq m < M$
\[
	\tau_m^{2(\nu_m\wedge r)} \geq C_\tau \,C_h^{\nu_m\wedge 1}  = C_\tau \, \hat \tau^{2(\nu_m\wedge 1)}_m, \qquad \forall \nu_m\in\bbN_0
\]
since $\hat \tau_m = \sqrt{C_h}$ for $m<M$. 
We now define
\begin{equation} \label{equ:d_nu}
	\hat c_\vnu^2 := C_\tau^{-M}\prod_{m\geq 1} \hat \tau_m^{- 2(\nu_m\wedge 1)} \hat h(\nu_m).
\end{equation}
and notice that $\hat c^2_\vnu$ dominates $\frac {c^2_\vnu}{b_\vnu}$ by \eqref{equ:proof_chen_star}.

\paragraph{Step 2: Show that $(\hat c_\vnu)_{\vnu \in \mc F}\in\ell^p(\mc F)$}
As for the $p$-summability, there holds
\begin{align*}
	\sum_{\vnu \in \mc F} \hat c_\vnu^{p} 	
	&=
	C_\tau^{-pM/2}\sum_{\vnu \in \mc F} \prod_{m\geq1} \hat \tau_m^{- p (\nu_m\wedge 1)} \hat h^{p/2} (\nu_m)\\	
	& =C_\tau^{-pM/2} \prod_{m\geq1} \sum_{n\geq 0} \hat \tau_m^{- p (n\wedge 1)} \hat h^{p/2} (n).
\end{align*}
We get
\begin{align*}
	\sum_{n\geq 0} \hat \tau_m^{- p (n\wedge 1)} \hat h^{p/2} (n)
	= 1 + C^{p/2}_h \hat \tau_m^{-p} \underbrace{\sum_{n \geq 1} n^{-p(r-2\theta-2)/2}}_{=:S}
\end{align*}
where the sum $S$ is finite due to the assumption $\frac p2 (r-2\theta-2) =  \frac p2 (r-2\theta-2) > 1$.
The rest follows by using $\log(1+x)\leq x$ for $x$ positive in order to get
\[
	\sum_{\vnu \in \mc F} \hat c_\vnu^{p} 
	=
	C_\tau^{-pM/2}\prod_{m\geq1} (1+ C^{p/2}_hS \hat \tau_m^{-p})
	\leq C_\tau^{-pM/2}\exp\left( C^{p/2}_hS \sum_{m \geq 1} \hat \tau_m^{-p}\right) < \infty
\]
since $(\hat \tau^{-1}_m)_{m\in\bbN}$ is in $\ell^p(\bbN)$ by construction.

\paragraph{Step 3: Show that $(\hat c_\vnu)_{\vnu \in \mc F}$ is nonincreasing}
Let $\vnu \in \mc F$ be arbitrary.
If $m \in \supp \vnu =\{m \in \bbN: \nu_m >0\}$, then we get
\[
	\hat c_{\vnu+\ve_m}^2 = \hat c_{\vnu}^2 \cdot \frac{\hat h(\nu_m+1)}{\hat h(\nu_m)}
	\leq 
	\hat c_{\vnu}^2,
\]
since $\hat h(n)$ is nonincreasing for $n\geq 1$.
Let now $m \notin \supp \vnu$. Then
\[
	\hat c_{\vnu+\ve_m}^2 
	= 
	\hat c_{\vnu}^2 \cdot \hat \tau^{-2}_m \cdot \hat h(1) 
	= 
	\hat c_{\vnu}^2 \cdot C_h \hat \tau^{-2}_m 
	\leq 
	\hat c_{\vnu}^2 \cdot C_h (\sqrt{C_h})^{-2} 
	\leq 
	\hat c_\vnu^2.
\]
In summary, we obtain
\[
	\hat c_{\vnu+\ve_m} \leq \hat c_{\vnu} \qquad \forall m\in\bbN,
\]
hence, $(\hat c_\vnu)_{\vnu\in\mc F}$ is nonincreasing.
\end{proof}

We can now state our main convergence result for sparse collocation.

\begin{theo}[Convergence of sparse collocation] \label{cor:conv}
Assume that for $\theta \geq 0$ and $K \geq 1$
there holds
\begin{equation} \label{equ:Delta_bound}
	\|\Delta_i H_\nu\|_{L^2_\mu} \leq (1+ K\nu)^\theta,
	\qquad i \in\bbN_0.
\end{equation}
Then, for any function $f$ which satisfies Assumption \ref{assum:df} with $r  > 2(\theta+1) + \frac 2p$ and Assumption \ref{assum:Xi}, there exists a nested sequence of monotone finite subsets $\Lambda_N\subset \mc F$ with $|\Lambda_N| = N$ such that for the sparse collocation error holds
\[
	\left\| f - U_{\Lambda_N} f \right\|_{L^2_\mu} \leq C (1+N)^{-\left( \frac 1p - \frac 12\right )}.
\]
\end{theo}
\begin{proof}
We prove the assertion by verifying the assumptions of Theorem \ref{theo:conv2}.
Since $f$ satisfies Assumption \ref{assum:df} with $r  > 2(\theta+1) + \frac 2p$, condition \eqref{equ:sum_b_nu_f_nu} of Theorem \ref{theo:conv2} holds due to Theorem \ref{theo:Bachmayr1}.
Moreover, we can apply Lemma \ref{lem:c_nu_b_nu} to verify the remaining assumption of Theorem \ref{theo:conv2} about a nonincreasing dominating sequence $(\hat c_\vnu)_{\vnu\in\mc F}\in\ell^p(\mc F)$, $p\in(0,2)$: due to Proposition \ref{thm:c_nu}) the bound \eqref{equ:Delta_bound} implies
\[
	c_\vnu \leq \prod_{m\geq1}^\infty (1+ K \nu_m)^{\theta+1}, \qquad \vnu\in\mc F,
\]
and the sequence $(\tau_m)_{m\in\bbN}$ appearing in Assumption \ref{assum:df} can w.l.o.g.~be assumed to be increasing (otherwise we can permute the dimension accordingly).
\end{proof}

%
%

\subsection{Convergence of Sparse Collocation Using Gauss-Hermite Nodes}\label{subsec:convGH}
In the following, we will verify the assumption \eqref{equ:Delta_bound} in Theorem~\ref{cor:conv} for the interpolation operators $U_i$ based on Gauss-Hermite nodes.
Moreover, we bound the number of sparse grid points $|\Xi_{\Lambda_N}|$ associated with a multi-index set $\Lambda_N$ allowing us to relate the convergence rate previously derived to a quantity which reflects the computational effort of the collocation approximation.
For nested univariate node sets, i.e., when $\Xi_{i+i} = \Xi_i \cup \{\xi^{(i+1)}_{i+1}\}$, we have $|\Xi_{\Lambda_N}| = |\Lambda_N|$.
This simple relation, however, fails to hold for non-nested interpolation sequences such as those based on Gauss-Hermite nodes.

\begin{lem} \label{lem:bound_UiHnu}
For $U_i$ being the interpolation operator based on the zeros of the $(i+1)$th Hermite polynomial we have for each $\nu \in \bbN$ that
\[
	\|U_i H_\nu \|^2_{L^2_\mu} \leq c^2\e \sqrt{2\nu-1} \qquad \forall i\in\bbN_0
\]
where $c = 1.086435$ is the constant appearing in \emph{Cram\'er's inequality} for Hermite functions.
In particular, there holds
\[
	\|\Delta_i H_\nu\|_{L^2_\mu} \leq (1+ K\nu)
\]
with $K = 2c\sqrt{\e}>1$.
\end{lem}
\begin{proof}
We start by recalling the $L^2_\mu$-orthogonality ($\mu$ refers here to the univariate standard Gaussian measure $N(0,1)$) of Lagrange basis polynomials $L_{k}^{(i)}$ constructed from the zeros $\{\xi_{k}^{(i)}\}_{k=0}^i$ of the Hermite polynomial of degree $i+1$ (\cite[Theorem 14.2.1]{Szego1939}).
This orthogonality yields
\begin{align*}
	\|U_i H_\nu \|^2_{L^2_\mu}
	&= 
	\int_{\bbR} \left(\sum_{k=0}^i H_\nu(\xi_k^{(i)}) L_k^{(i)}(\xi)\right)^2 \, \mu(\d \xi)
	= 
	\sum_{k=0}^i  H^2_\nu(\xi_{k}^{(i)}) 
	\int_{\bbR} \left(L_{k}^{(i)}(\xi)\right)^2 \, \mu(\d \xi)\\
	& = 
	\sum_{k=0}^i  H^2_\nu(\xi_{k}^{(i)}) w_{k}^{(i)}
\end{align*}
where $\{w_{k}^{(i)}\}_{k=0}^i$ denotes the weights of the Gauss quadrature formulae based on the zeros of the $(i+1)$th Hermite polynomial, see also \cite[Theorem 14.2.1]{Szego1939}.

Next, we recall Cram\'er's inequality for the Hermite polynomials $\tilde H_\nu$ taken w.r.t. the weight function $\tilde \rho(\xi) = \exp(-\xi^2)$, i.e.,
\[
	|\tilde H_n(\xi)| \leq c \pi^{-1/4} \exp(\xi^2/2),
\]
see, e.g., \cite[Chapter 22, p.787 ]{AbramowitzStegun1972}.
Since there holds $\tilde H_n(\xi) = \pi^{-1/4} H_n(\xi \sqrt 2)$ \cite[Chapter 22, p.778 ]{AbramowitzStegun1972}, we get
\[
	|H_n(\xi)| \leq c \exp(\xi^2/4)
\]
and, thus,
\begin{align*}
	\|U_i H_\nu \|^2_{L^2_\mu}
	& \leq c^2 \sum_{k=0}^i  \exp(\xi^2_{ki}/2) w_{ki},
\end{align*}
where we switched notation to $\xi_{ki}:=\xi_{k}^{(i)}$ and $w_{ki} := w_{k}^{(i)}$ for convenience.
Furthermore, we use a consequence of \cite[Lemma 4]{Nevai1980}.
The latter states for $\tilde \xi_{kn}$ being the zeros of $\tilde H_n$ and $\tilde w_{kn}$ the Christoffel numbers of corresponding Gauss-Hermite quadrature (i.e. Gauss-Hermite weights for $\tilde \rho$) that
\[
	\sum_{k=1}^n \tilde w_{kn} \exp(\tilde \xi^2_{kn}) \leq \e \sqrt{\pi(2n+1)}.
\]
It can be easily verified that
\[
	\xi_{kn} = \sqrt 2 \tilde \xi_{kn} 
	\quad \text{ and }\quad
	w_{kn} = \pi^{-1/2} \tilde w_{kn}.
\]
Hence, we get
\begin{align*}
	\sum_{k=0}^i  \exp(\xi^2_{ki}/2) w_{ki} 
	\leq \e \sqrt{2(i+1)+1}
\end{align*}
and by noticing that for $i \geq \nu$ we have $U_i H_\nu = H_\nu$ and, thus, $\|U_i H_\nu \|^2_{L^2_\mu}=1$, and for $i=\nu-1$ we get $U_i H_\nu \equiv 0$ the first assertion is shown.

For the second statement we notice 
\[
	\|U_i H_\nu \|^2_{L^2_\mu} \leq c^2\e \nu,	\qquad \forall i \in \bbN_0\, \forall \nu \geq 1
\]
since $\nu \geq \sqrt{2\nu-1}$ for $\nu \geq 1$.
And, because of $\Delta_i H_0 \equiv 0$ for $i\geq 1$ and $\Delta_0 H_0 \equiv H_0$, we get
\[
	\|\Delta_i H_\nu \|_{L^2_\mu} \leq 1+ K \nu, \qquad \forall i,\nu \in \bbN_0.
\]
\end{proof}

Thus, interpolation on Gauss-Hermite points satisfies the assumptions of Theorem~\ref{cor:conv} with $\theta=1$ and we obtain

\begin{theo}[Convergence of sparse collocation, Gauss--Hermite nodes] \label{cor:conv_GH}
For any function $f$ which satisfies Assumption \ref{assum:df} with $r  > 4 + \frac 2p$ and Assumption~\ref{assum:Xi}, there exists a nested sequence of mononote finite subsets $\Lambda_N\subset \mc F$ with $|\Lambda_N| = N$ such that for the error of the sparse collocation operator $U_{\Lambda_N}$ based on Gauss-Hermite nodes holds
\[
	\left\| f - U_{\Lambda_N} f \right\|_{L^2_\mu} \leq C (1+N)^{-\left( \frac 1p - \frac 12\right )}.
\]
\end{theo}

\begin{rem}
In numerical experiments we actually observed for $\nu=0,\ldots,39$ that
\[
	\|U_i H_\nu \|_{L^2_\mu} \leq 1,	\qquad \forall i \in \bbN_0,
\]
see Figure \ref{fig:UpGH}.
This would imply
\[
	\|\Delta_i H_\nu \|_{L^2_\mu} 
	\leq 
	\begin{cases}
	   1 & \text{ if } \nu = 0,\\
	   2 & \text{ otherwise, }
	\end{cases} 
	\qquad \forall i,\nu \in \bbN_0.
\]
Again, we even observed a smaller bound numerically, see the right plot in Figure \ref{fig:UpGH}.
However, we have not been able to prove $\|U_i H_\nu \|_{L^2_\mu} \leq 1$ and the improvement in the statement of Theorem~\ref{cor:conv_GH} would have been minor, i.e., the assertion would also hold with the same rate for functions $f:\Gamma\to\mc H$ satisfying Assumption \ref{assum:df} with $r  > 2 + \frac 2p$. Note that similar numerical evidence was presented in \cite{Chen2016} for quadrature operators applied to
Hermite polynomials. See also \cite{back.nobile.eal:lognormal} for analogous numerical bounds in the case of Genz-Keister points. 
\end{rem}

\begin{figure}[h]
\hfill
\begin{minipage}{0.49\textwidth}
\includegraphics[width = \textwidth]{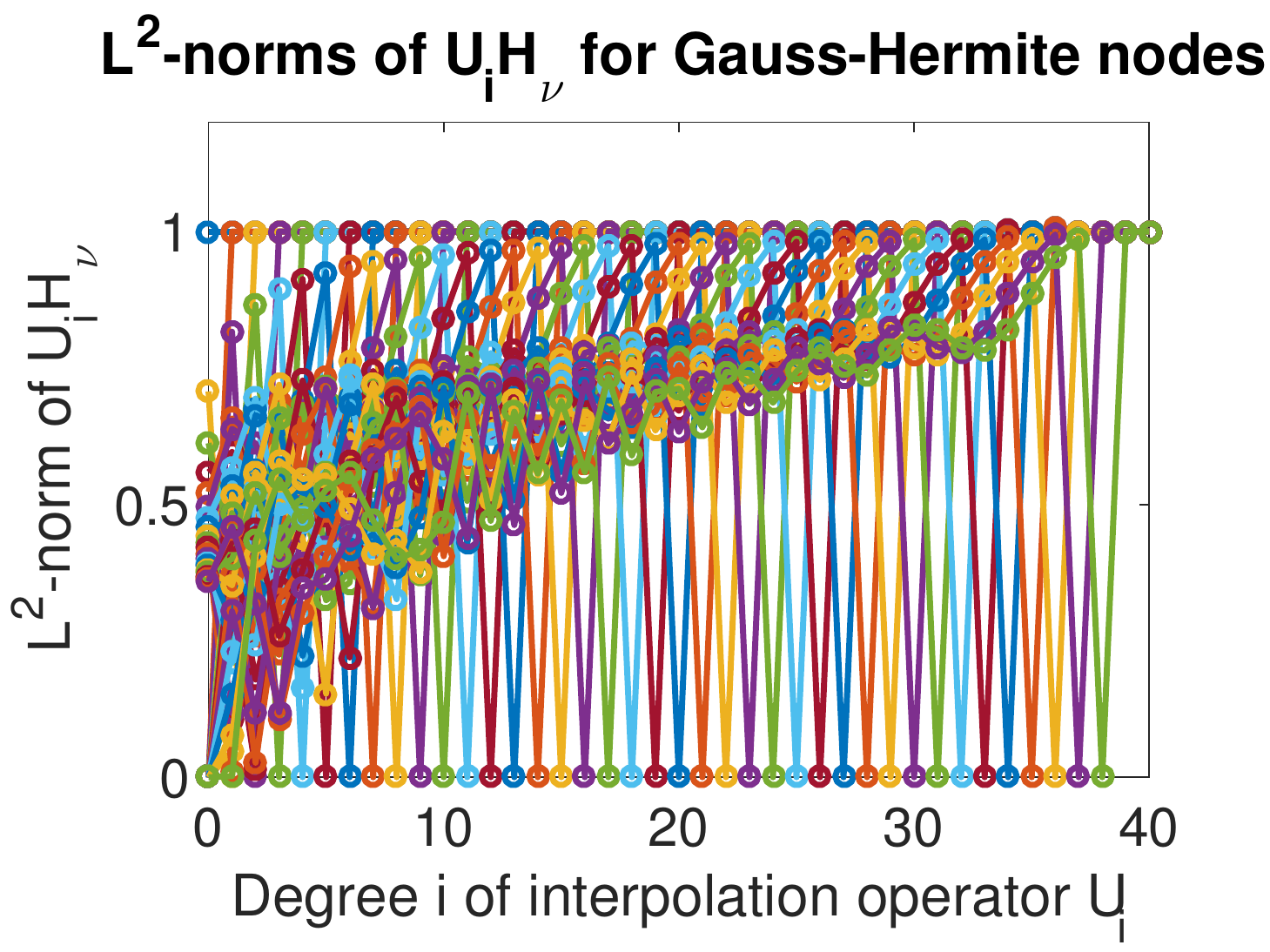}\\[3pt]
\end{minipage}
\hfill
\begin{minipage}{0.49\textwidth}
\includegraphics[width = \textwidth]{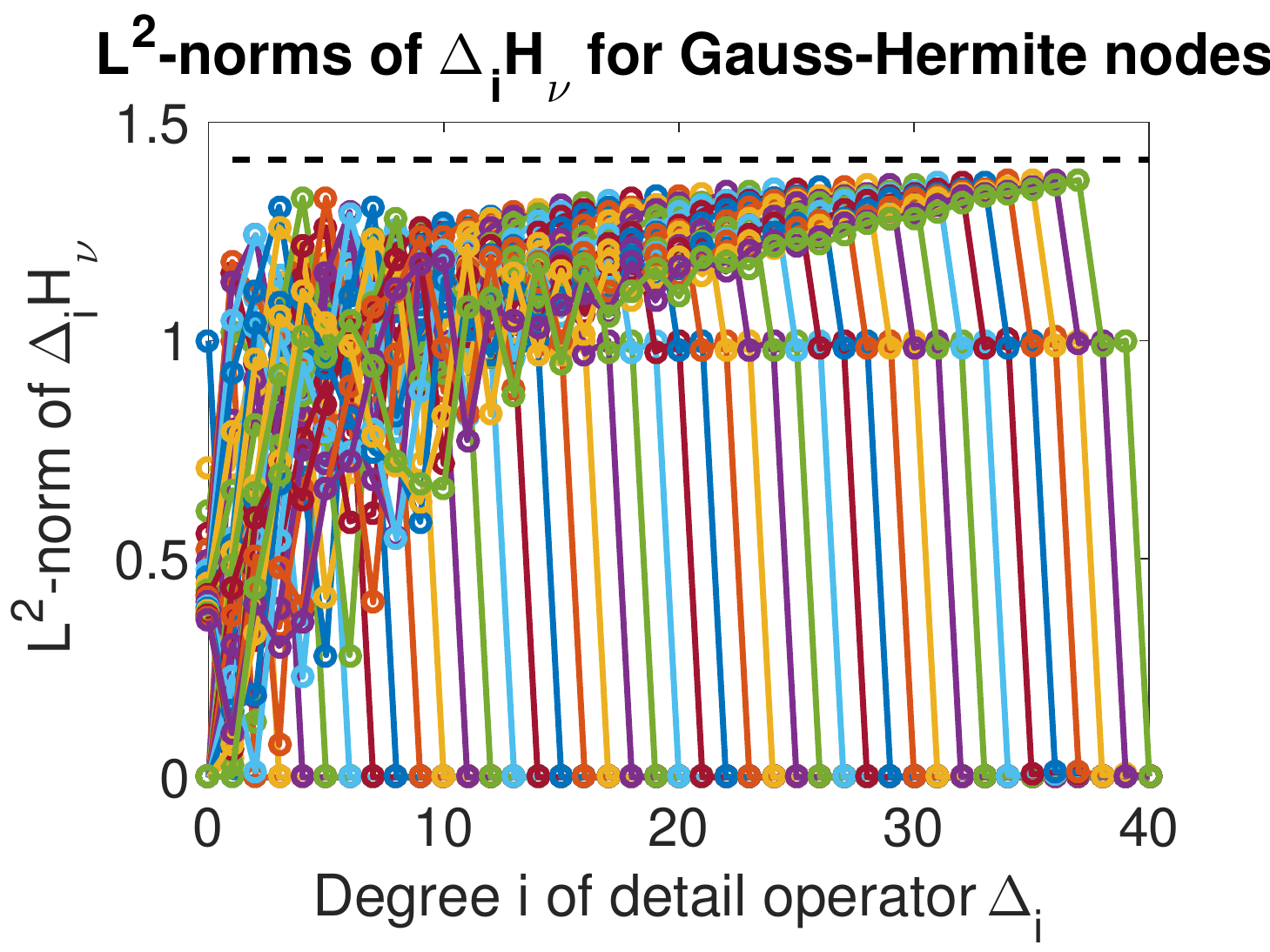}\\[3pt]
\end{minipage}
\label{fig:DeltaiGH}
\caption{Computed values of $\|U_iH_{\nu}\|_{L^2_\mu}$ (left) and $\|\Delta_iH_{\nu}\|_{L^2_\mu}$ (right) for Gauss-Hermite nodes. 
The dashed, black line in the right plot indicates the value $\sqrt2$.}
\label{fig:UpGH}
\end{figure} 

%
%

\subsection{Convergence Rate With Respect to Number of Collocation Nodes}

We now derive bounds for the number of nodes in the sparse grid $\Xi_\Lambda$ associated with $U_\Lambda$.
Consider first the following simple monotone index of cardinality $N$: $\Lambda_N = \{0\ve_j,\ldots, (N-1)\ve_j\}$.
Then due to $|\Xi_{\{k\ve_j\}}| = (k+1)$ we get for this $\Lambda_N$ that
\[
	|\Xi_{\Lambda_N}| \leq \sum_{k=0}^{N-1} (k+1) = \frac{N(N+1)}{2} \in \mc O(N^2).
\]
The quadratic complexity is indeed sharp, since $0$ is the only reappearing Gauss-Hermite node.
We will show in the subsequent two propositions that this complexity holds also for arbitrary monotone multi-index sets. 
We start with rectangular envelopes $\Lambda = \mc R_\vnu$ and provide also a rather technical ordering result which we will require later on.
Recall that $|\Xi^{(\vi)}| = \prod_{m\geq 1} (1+i_m)$.

\begin{propo} \label{propo:X_R_vnu}
Let $\vnu\in\mc F$. 
Then there exists an ordering $n$ of $\mc R_\vnu$, i.e., a bijective mapping $n:\mc R_\vnu \to \{1, \ldots, |\mc R_\vnu|\}$ such that
\[
	|\Xi^{(\vi)}| \leq n(\vi) \qquad \forall \vi \in \mc R_\vnu,
\]
which implies, in particular,
\[
	\left| \Xi_{\mc R_\vnu} \right| 
	\leq \frac{|\mc R_\vnu|\,(|\mc R_\vnu|+1)}{2}.
\] 
\end{propo}
\begin{proof}
The second assertion follows easily by the first one since
\[
	\left| \Xi_{\mc R_\vnu} \right| 
	\leq \sum_{\vi \leq \vnu} |\Xi^{(\vi)}| 
	\leq \sum_{n=1}^{|\mc R_\vnu|} n
	= \frac{|\mc R_\vnu|\,(|\mc R_\vnu|+1)}{2}.
\] 
We prove the first assertion by induction.
Since $\vnu\in\mc F$, there exist only finitely many $m\in\bbN$ such that $\nu_m > 0$.
Without loss of generality we assume that $\nu_m = 0$ for $m > M$ where $M\in\bbN$.
We now perform an induction over the number $M$ of non-zero entries in $\nu$.

\begin{itemize}
\item
\textbf{base case} $M=1$: 
The only possible multi-indices $\vnu\in\mc F$ are $\vnu = k\ve_1$, $k\in\bbN_0$, and we have $R_\vnu = \{0 \ve_1, 1\ve_1, \ldots, \nu_1 \ve_1\}$.
The ordering is then simply $n(\vi) = i_1 + 1$.
Then
\[
	|\Xi^{((i_1,0,\ldots))}| = 1+i_1 = n(\vi).
\]

\item
\textbf{Induction step:} the assertion holds for $M\geq1$.
Let $\vnu \in \mc F$ be such that $\nu_m = 0$ for $m\geq M+2$.
Moreover, let $n_M$ denote the ordering for $\mc R_{\vnu - \nu_{M+1} \ve_{M+1}} = \{\vi \in \mc R_\vnu: i_{M+1} = 0\}$, i.e., it holds
\[
	|\Xi^{(\vi)}| = \prod_{m=1}^{M} (1+i_m) \leq n_M( \vi ) \qquad \forall \vi \in \mc R_{\vnu - \nu_{M+1} \ve_{M+1}}.
\]
For notational convenience, we set $\vi_M := (i_1,\ldots,i_M,0,\ldots)$ for each $\vi \in \mc R_\vnu$ and 
observe that $\vi_M \in \mc R_{\vnu - \nu_{M+1} \ve_{M+1}}$. We define the ordering
\[
	n(\vi) := i_{M+1}\,\left(\prod_{m=1}^M (1+\nu_m)\right) + n_M(\vi_M), \qquad \vi\in\mc R_\vnu.
\]
It is easy to check that $n:\mc R_\vnu \to \{1,\ldots, |\mc R_\vnu|\}$ is again bijective.
Furthermore, we get for each $\vi \in \mc R_\vnu$
\begin{align*}
	|\Xi^{(\vi)}| 
	& = \prod_{m=1}^{M+1} (1+i_m)
	= (i_{M+1}+1)\, \prod_{m=1}^M (1+i_m)\\
	& = i_{M+1} \left(\prod_{m=1}^M (1+i_m)\right) + \prod_{m=1}^M (1+i_m)\\
	& = i_{M+1} \left(\prod_{m=1}^M (1+i_m)\right) + |\Xi_{\vi_m}|\\
	& \leq i_{M+1} \left(\prod_{m=1}^M (1+\nu_m)\right) + n_M(\vi_M) = n(\vi)
\end{align*}
where the last line follows by $i_m \leq \nu_m$ for all $m\geq 1$ and the fact that $\vi_M\in \mc R_{\vnu - \nu_{M+1} \ve_{M+1}}$ for $\vi \in \mc R_\vnu$.
\end{itemize}
\end{proof}

We extend the estimate for $\Xi_{\mc R_\vnu}$ in the above proposition now to arbitrary finite and monotone index sets $\Lambda$:

\begin{propo} \label{propo:X_Lambda}
Let $\Lambda \subset \mc F$ be a finite and monotone, then there holds
\begin{equation}\label{eq:nb_pts_growth}
	|\Xi_\Lambda| \leq \frac{|\Lambda|\,(|\Lambda|+1)}{2}.  
\end{equation}
\end{propo}
\begin{proof}
Since $\Lambda$ is supposed to be monotone, it is a union of rectangular envelopes, see Proposition \ref{propo:montone_set}.
Thus, there exist $n$ indices $\vnu_1,\ldots, \vnu_n \in \mc F$ such that
\[
	\Lambda = \bigcup_{k=1}^n \mc R_{\vnu_k} 
	\quad \text{ and } \quad 	\Xi_{\Lambda} = \bigcup_{k=1}^n \Xi_{\mc R_{\vnu_k}}.
\]
We prove the assertion by induction over $n$:
\begin{itemize}
\item
\textbf{base case} $n=1$: 
The assertion follows by Proposition \ref{propo:X_R_vnu}.

\item
\textbf{Induction step:} the assertion holds for $n\geq 1$.
With a slight abuse of notation we set $\Lambda_n := \bigcup_{k=1}^n \mc R_{\vnu_k}$ and obtain
\[
	\sum_{\vi \in \Lambda_n \cup \mc R_{\vnu_{n+1}}} |\Xi^{(\vi)}|
	= 
	\sum_{\vi \in \Lambda_n} |\Xi^{(\vi)}|
	+ \sum_{\vi \in \mc R_{\vnu_{n+1}}\setminus \Lambda_n} |\Xi^{(\vi)}|.
\]
Let $m:=|\mc R_{\vnu_{n+1}}\setminus \Lambda_n|$.
The first statement of Proposition \ref{propo:X_R_vnu} now implies
\[
	\sum_{\vi \in \mc R_{\vnu_{n+1}}\setminus \Lambda_n} |\Xi^{(\vi)}|
	\leq \sum_{k = 1+|\mc R_{\vnu_{n+1}}| - |\mc R_{\vnu_{n+1}}\setminus \Lambda_n|}^{|\mc R_{\vnu_{n+1}}|} k
	\leq \sum_{k = 1+|\Lambda_n|}^{|\Lambda_n|+|\mc R_{\vnu_{n+1}}\setminus \Lambda_n|} k
\]
where the last inequality is due to $|\mc R_{\vnu_{n+1}}| \leq |\mc R_{\vnu_{n+1}}\setminus\Lambda_n| + |\Lambda_n|$.
Thus, we get by the induction hypothesis
\[
	\sum_{\vi \in \Lambda_n \cup \mc R_{\vnu_{n+1}}} |\Xi^{(\vi)}|
	\leq \sum_{k=1}^{|\Lambda_n|} k + \sum_{k = 1+|\Lambda_n|}^{|\Lambda_n|+|\mc R_{\vnu_{n+1}}\setminus \Lambda_n|} k
	= \sum_{k = 1}^{|\Lambda_n \cup \mc R_{\vnu_{n+1}}|} k.
\]
\end{itemize}
\end{proof}

Thus, employing non-nested points such as Gauss-Hermite points, yields at most a quadratic growth of the number of sparse grid points 
\[
	|\Xi_\Lambda| \in \mc O(|\Lambda|^2)
\]
whereas in the nested case one has $|\Xi_\Lambda| = |\Lambda|$.
\begin{rem}
We provide some numerical validation of the bound \eqref{eq:nb_pts_growth}. 
More precisely, we consider the following two families of multi-index sets $\Lambda$ (cf. \cite{back.nobile.eal:comparison}):
\begin{description}
	\item[Total Degree (TD):] $\Lambda = \Lambda(w,M) = \{\vnu \in \mc F: \sum_{m=1}^M \nu_m \leq w,\; \nu_m = 0 \text{ for } m > M\}$ 
	\item[Hyperbolic Cross (HC):] $\Lambda = \Lambda(w,M) = \{\vnu \in \mc F : \prod_{m=1}^M (\nu_m+1) \leq w,\; \nu_m = 0 \text{ for } m > M\}$,
\end{description}
In Figure \ref{fig:check_bound_nb_pts} we fix the number of (active) dimensions $M$ and display the cardinality of $\Xi_{\Lambda(w,M)}$ for both choices of $\Lambda(w,M)$ and increasing values of $w \in \bbN$.
The plot shows that estimate \eqref{eq:nb_pts_growth} is valid but slightly pessimistic for the two specific examples considered here.
\begin{figure}[h]
  \centering
  \includegraphics[width = 0.49\textwidth]{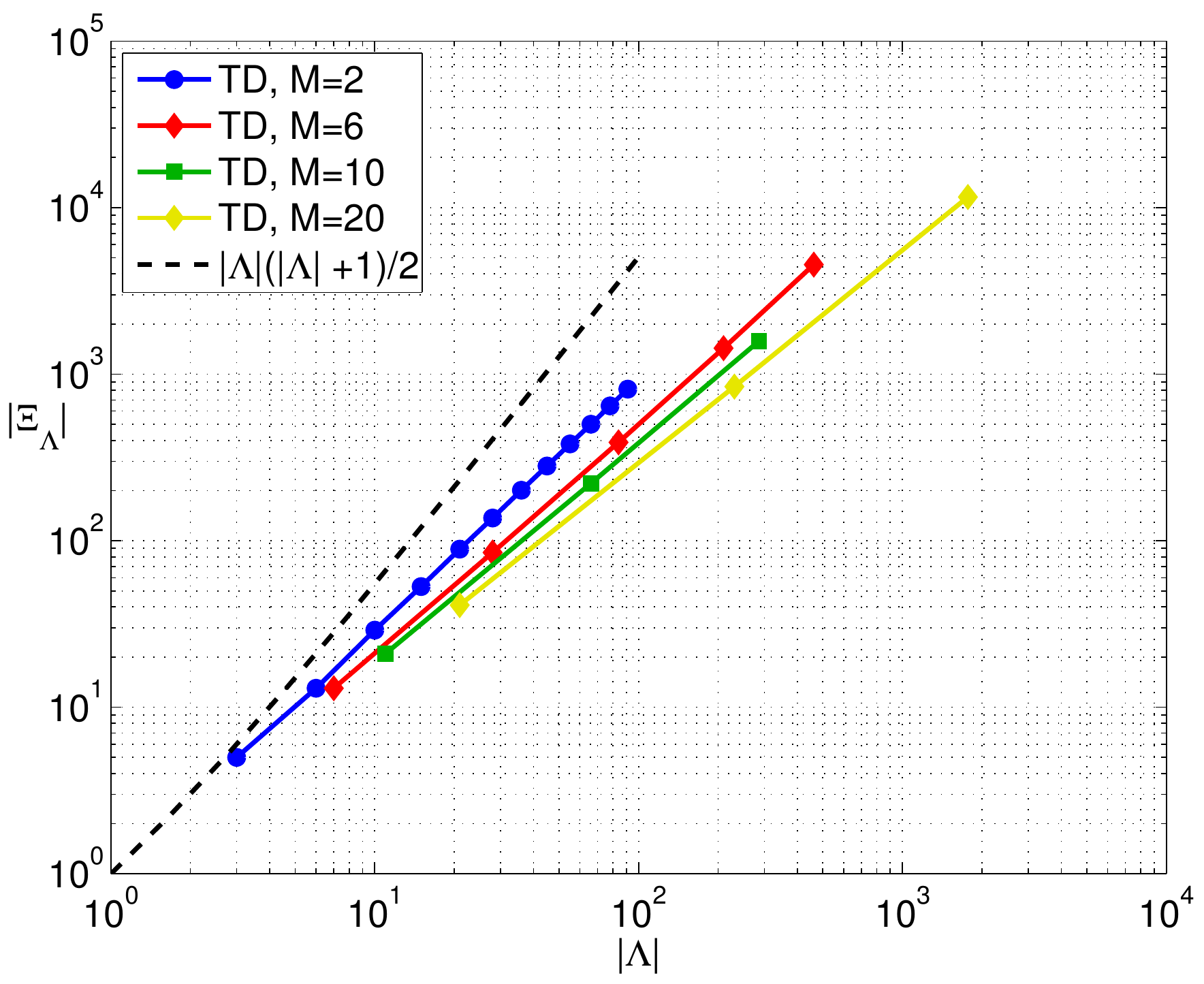}
  \includegraphics[width = 0.49\textwidth]{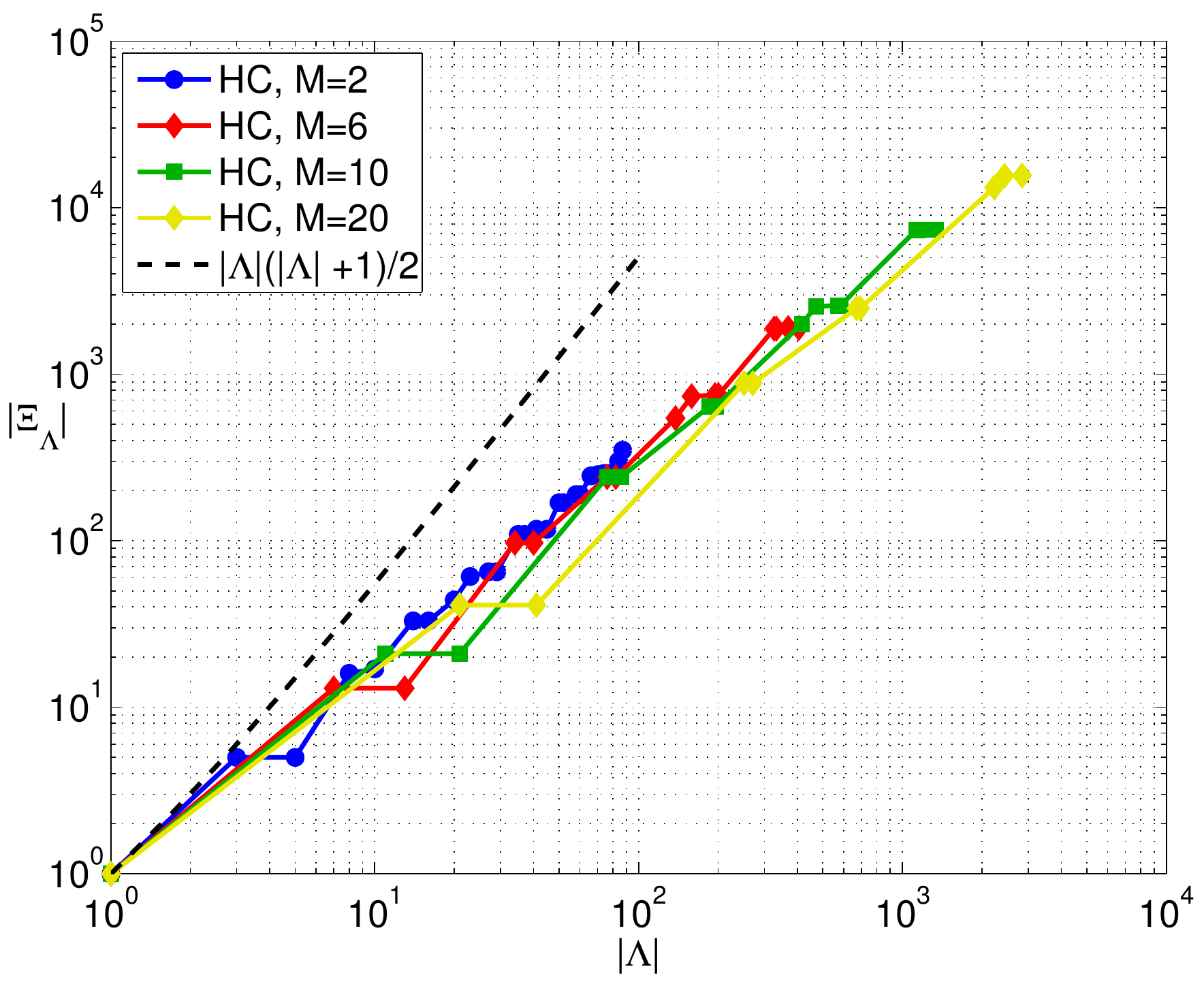}
  \caption{Numerical verification of estimate \eqref{eq:nb_pts_growth} for ``Total Degree'' sparse grids (left) and ``Hyperbolic Cross'' sparse grids (right).}
  \label{fig:check_bound_nb_pts}
\end{figure}
\end{rem}

We finally arrive at the resulting error-cost theorem: 

\begin{theo}[Convergence rate of Gauss-Hermite sparse grid collocation in terms of nodes]\label{theo:GH-conv-vs-nodes}
For any function $f$ which satisfies Assumption \ref{assum:df} with $r  > 4 + \frac 2p$ and Assumption~\ref{assum:Xi}, there exists a nested sequence of mononote finite subsets $\Lambda_N\subset \mc F$ with $|\Lambda_N| = N$ such that for the error of the sparse collocation operator $U_{\Lambda_N}$ based on Gauss-Hermite nodes holds
\[
	\|f - U_{\Lambda_N} f\|_{L^2_\mu} \leq C |\Xi_{\Lambda_N}|^{-\left(\frac 1{2p} - \frac 14\right)}
\]
where $C$ depends on $f$.
\end{theo}

Hence, assume we require an approximation error $\|f - U_{\Lambda_N} f\|_{L^2_\mu} \leq \varepsilon$, then we can achieve this accuracy with
\begin{equation}\label{equ:cost}
	\mathrm{cost}(\varepsilon) \in \mc O\left( \varepsilon^{\frac 1{2p} - \frac 14} \right)
\end{equation}
number of function evalutions of $f$.
In this cost complexity \eqref{equ:cost} we neglected of course the computational work which is necessary to find the resulting multi-index sets $\Lambda_N$.
This is a very important issue. 
Typically, they are constructed employing adaptive algorithms, see \cite{ChkifaEtAl2014, SchillingsSchwab2013, NobileEtAl2016} and also Section \ref{sec:numerics}.
Our result makes no statement about the actual computational work of those.

\begin{rem}[On sparse collocation employing weighted Leja points]
As mentioned in the introduction weighted Leja points \cite{NarayanJakeman2014} seem to be a promising node family for interpolation and sparse collocation.
So far we are, however, unable to prove bounds like \eqref{equ:Delta_bound} for them.
Possibly a more suitable approach for analyzing convergence in case of weighted Leja nodes is to measure the approximation error in the $L^\infty_\mu$-norm instead of the $L^2_\mu$-norm and to estimate the corresponding Lebesgue constant.
See \cite{JantschEtAl2016} for first results on the latter -- which does not yet imply an analogous estimate to \eqref{equ:Delta_bound} -- and \cite{ChkifaEtAl2015,ChkifaEtAl2014} for the convergence analysis of sparse collocation using Leja points on $[-1,1]$ via estimates of the associated Lebesgue constant \cite{Chkifa:leja}.
\end{rem}

%
%
\section{Application to Elliptic PDEs}\label{sec:PDE}
We recall our motivation from the introduction: approximating the weak solution $u$ of an elliptic boundary value problem with lognormal diffusion coefficient as in \eqref{equ:PDE} where $f\in L^ 2(D)$ and $a(\vxi)\in L^\infty(D)$ is given as \eqref{equ:KLE}.
We will discuss now under which conditions the map $\vxi \mapsto u(\vxi) \in H_0^1(D)$ satisfies Assumptions \ref{assum:f}, \ref{assum:Xi} and \ref{assum:df} and can therefore be approximated by sparse grid collocation methods based on Gauss-Hermite nodes as outlined in the previous section.
We will mainly cite results from \cite{BachmayrEtAl2015} but try to emphasize those details which are sometimes omitted in the literature.

\paragraph{Verifying Assumption \ref{assum:f}}
First of all, we have to investigate the domain $\Gamma$ of the mapping $\vxi \mapsto u(\vxi)$.
There holds $\Gamma \neq \bbR^\bbN$ since for arbitrary $\vxi \in \bbR^\bbN$ the expansion \eqref{equ:KLE} need not converge.
A natural domain $\Gamma$ for the mapping $\vxi \mapsto u(\vxi)$ is
\begin{equation}\label{equ:Gamma}
	\Gamma := \{\vxi\in\bbR^\bbN:  \| \sum_{m=1}^\infty \phi_m \xi_m\|_{L^\infty(D)} < \infty \}.
\end{equation}
Further, a natural condition on the decay of the $\phi_m$ is
\begin{equation} \label{equ:KL_modes}
	\sum_{m=1}^\infty	\|\phi_m\|^2_{L^\infty(D)} < \infty
\end{equation}
since \eqref{equ:KL_modes} implies that the series \eqref{equ:KLE} converges in $L^2_\mu(\bbR^\bbN; L^\infty(D))$:
\begin{align*}
	\ev{\left\| \sum_{m=1}^\infty \phi_m \, \xi_m \right\|^2_{L^\infty(D)}}
	& \leq
	\ev{ \left( \sum_{m=1}^\infty  \left\| \phi_m(x)\right\|_{L^\infty(D)} |\xi_m| \right)^2}\\
	& = \sum_{m=1}^\infty  \left\| \phi_m(x)\right\|^2_{L^\infty(D)} \ev{|\xi_m|^2}
	= \sum_{m=1}^\infty  \left\| \phi_m(x)\right\|^2_{L^\infty(D)}
\end{align*}
due to $\ev{\xi_m\xi_n} = \delta_{mn}$.
Moreover, by a classical result \cite[Lemma 4.16]{Kallenberg2002} 
from probability theory this implies that the series converges also $\mu$-a.e. in $L^\infty(D)$.
Thus, if \eqref{equ:KL_modes} holds, then we get $\mu(\Gamma) = 1$.
It remains to state conditions under which we can ensure that $\vxi \mapsto u(\vxi)$ belongs to $L^2_\mu(\Gamma; H_0^1(D))$.
Measurability follows from the continuous dependence of the weak solution $u\in H_0^1(D)$ on $\exp(a)\in L^\infty(D)$, see \cite{GilbargTrudinger2001}.
Moreover, if we can ensure that for $\underline a(\vxi) := \essinf_{x\in D} \exp(a(x,\vxi))$ 
we have $\underline a^{-1}\in L^2_\mu(\Gamma; \bbR)$ 
(e.g., by Fernique's lemma, as shown in \cite{Charrier2012})
then the $\vxi$-pointwise application of the Lax-Milgram lemma \cite{GilbargTrudinger2001} 
yields for the random solution $u\in L^2_\mu(\Gamma; H_0^1(D))$.
The latter can be guaranteed by an even weaker assumption than \eqref{equ:KL_modes}

\begin{assum}[{\cite[Assumption A]{BachmayrEtAl2015}}]\label{assum:KL_1}
There exists a strictly positive sequence $(\tau_m)_{m\in\bbN}$ such that
\begin{align*}
	\sup_{x\in D} \sum_{m=1}^\infty \tau_m |\phi_m(x)| < \infty, \qquad 
	\sum_{m=1}^\infty \exp(-\tau^2_m) < \infty.
\end{align*}
\end{assum}

Under Assumption~\ref{assum:KL_1}, it is shown in 
\cite[Corollary 2.1]{BachmayrEtAl2015} 
that $u\in L^2_\mu(\Gamma; H_0^1(D))$ with $\mu(\Gamma)=1$, hence $u:\Gamma\to H_0^1(D)$ satisfies Assumption \ref{assum:f}.

\paragraph{Verifying Assumption \ref{assum:Xi}} 
It is obvious that for Gauss-Hermite nodes there holds $\Xi^{(\vi)} \subset \Gamma$, $\vi \in \mc F$,  with $\Gamma$ as in \eqref{equ:Gamma}, because due to $\vi \in \mc F$ there exists an $M\in\bbN$ such that for $\vxi\in\Xi^{(\vi)}$ we have $\xi_m = \xi^{(0)}_0$ for any $m\geq M$ and $\xi^{(0)}_0=0$.
Actually, by Assumption \ref{assum:KL_1} there holds for any $\vxi \in \ell^\infty(\bbN)$ that $\vxi\in\Gamma$:
\begin{align*}
	\left\| \sum_{m=1}^\infty \phi_m \xi_m\right\|_{L^\infty(D)}
	& \leq \|\vxi\|_{\ell^\infty}\, \sup_{x\in D} \sum_{m=1}^\infty |\phi_m(x)|
	\leq \frac{\|\vxi\|_{\ell^\infty}}{\min_m \tau_m} \, \sup_{x\in D} \sum_{m=1}^\infty \tau_m\, |\phi_m(x)|
	<\infty,
\end{align*}
where $\min_m \tau_m > 0$, because Assumption \ref{assum:KL_1} implies $\tau_m\to \infty$ as $m\to \infty$.

\paragraph{Verifying Assumption \ref{assum:df}} Again, we refer to results from \cite{BachmayrEtAl2015}, namely, \cite[Theorem 4.2]{BachmayrEtAl2015} where the authors show that the (weak) solution $u$ of \eqref{equ:PDE} satisfies Assumption \ref{assum:df} for any $r\in\bbN_0$ given

\begin{assum}\label{assum:KL_2}
There exists a strictly positive sequence $(\tau^{-1}_m)_{m\in\bbN} \in \ell^p(\bbN)$ such that
\begin{align*}
	\sup_{x\in D} \sum_{m=1}^\infty \tau_m |\phi_m(x)| < \infty.
\end{align*}
\end{assum}

Please note that Assumption \ref{assum:KL_2} implies Assumption \ref{assum:KL_1}, see \cite[Remark 2.2]{BachmayrEtAl2015}.
Hence, we obtain

\begin{theo}\label{cor:conv-final-for-PDE}
Let $a$ be given as in \eqref{equ:KLE} and satisfy Assumption \ref{assum:KL_2}.
Then there exists a nested sequence of monotone finite subsets $\Lambda_N\subset \mc F$ with $|\Lambda_N| = N$ such that for the sparse collocation operator $U_{\Lambda_N}$ based on Gauss-Hermite nodes applied to the solution $u$ of \eqref{equ:PDE} holds 
\begin{align*}
	\|u - U_{\Lambda_N} u\|_{L^2_\mu(\bbR^\bbN; H_0^1(D))} 
	\; \leq \; C_1 N^{-\left(\frac 1{p} - \frac 12\right)}
	\; \leq \; C_2 |\Xi_{\Lambda_N}|^{-\left(\frac 1{2p} - \frac 14\right)}.
\end{align*}
\end{theo}

%
%

\section{Numerical Experiments}\label{sec:numerics}
We apply the sparse collocation outlined and analyzed in the previous sections to approximate the solution $u$ of a simple boundary value problem taken from \cite[Section 7]{BachmayrEtAl2015}.
In particular, we verify numerically the statement of Theorem~\ref{cor:conv-final-for-PDE} and provide some comments on algorithms for constructing sparse grid approximations.

%
%

\subsection{Problem Setting}
We consider the following boundary value problem on the unit interval $D=[0,1]$:
\begin{equation} \label{equ:BVP_1D}
	-\frac{\d}{\d x} \left(a(x,\vxi)\, \frac{\d}{\d x} u(x,\vxi)\right)
	= f(x),
	\quad
	u(0,\vxi) = u(1,\vxi) = 0,
	\qquad \mu\text{-a.e.}
\end{equation}
where we choose $f(x) = 0.03 \sin(2\pi x)$ and employ for $\log a$ the following expansion
\begin{equation} \label{equ:KLE_1D}
	\log a(x, \vxi) = 0.1 \sum_{m=1}^\infty \frac{\sqrt 2}{(\pi m)^	q} \, \sin(m\pi x) \xi_m, 
	\qquad
	\xi_m\sim N(0,1) \text{ i.i.d.}, \; q \geq 1.
\end{equation}
For $q=1$ the random field $\log a$ is a Brownian bridge, cf. \cite[Section 7]{BachmayrEtAl2015}, and for $q>1$ it is a smoother random field.
In particular, we get with $\phi_m(x) := \frac{\sqrt 2}{(\pi m)^q} \, \sin(m\pi x)$ that for $k = q-1 - \varepsilon$ with $\varepsilon >0$ 
\[
	\sup_{x\in D} \sum_{m\geq 1} m^{k} |\phi_m(x)|
	\leq \frac{\sqrt 2}{\pi^q} \sum_{m\geq 1} m^{- (q-k)}
	\propto \sum_{m\geq 1} m^{-(1+\varepsilon)}
	< \infty.
\]
Thus, given $q>1$ the expansion \eqref{equ:KLE_1D} satisfies Assumption \ref{assum:KL_2} for each $p > \frac 1{q-1}$ and 
according to Theorem~\ref{cor:conv-final-for-PDE} there exists,  if $q>1.5$, a nested sequence of monotone finite subsets $\Lambda_N\subset \mc F$, $|\Lambda_N| = N$, such that for the sparse collocation operator $U_{\Lambda_N}$ based on Gauss-Hermite nodes holds
\begin{equation} \label{equ:rate_1D}
	\|u - U_{\Lambda_N} u\|_{L^2_\mu(\bbR^\bbN; H_0^1(D))} 
	\; \leq \; C N^{-\left(q-1.5\right)}	
	\; \leq \; C |\Xi_{\Lambda_N}|^{-\left(\frac {q-1.5}{2}\right)}.
\end{equation}
In the following we will verify these rates numerically for various values of $q$.

%
%
\subsection{Numerical Algorithms}
The multi-index sets $\Lambda_N$ appearing in Theorem~\ref{cor:conv-final-for-PDE} and \eqref{equ:rate_1D} correspond to the largest entries in a $p$-summable decreasing sequence $(\hat c_\vnu)_{\vnu\in\mc F}$ which dominates $\left(c_\vnu/b^{1/2}_\vnu\right)_{\vnu\in\mc F}$.
Such a dominating sequence was constructed in Lemma \ref{lem:c_nu_b_nu}.
However, the resulting multi-index sets $\Lambda_N$ are in general not available as closed form expressions and need to be constructed by numerical algorithms.

\paragraph{A-priori algorithm}
The following greedy algorithm is based on \cite{GerstnerGriebel2003} and appears in a similar form in the recent work \cite{Chen2016}.
It successively adds to the set of multi-indices $\Lambda$ a new multi-index $\vnu$ from the set of neighbors $\mc N(\Lambda)$ which maximizes $|\hat c_{\vnu}|$.
A constraint $m_\text{buffer}$ restricts the index of dimensions considered for admissible neighbors:
\begin{enumerate}
\item
Initialize $N=1$ and $\tilde \Lambda_N := \{\boldsymbol 0\}$, choose $m_\text{buffer}\in\bbN$ and $N_{\max} \in \bbN$.

\item
For $N=2,\ldots,N_{\max}$ set
\begin{align} \label{equ:index_choice_prior}
	\tilde \Lambda_{N} := \tilde \Lambda_{N-1} \cup \{\vnu_{N}^*\}, 
	\qquad
	\vnu_{N}^* := \argmax_{\vnu \in \mc N(\tilde \Lambda_{N-1})} \left| \hat c_\vnu \right|
\end{align}
where with $\supp(\vnu) := \{m\in\bbN: \nu_m >0\}$ and $\supp(\Lambda) := \bigcup_{\vnu \in \Lambda} \supp(\vnu)$
\begin{align*}
	\mc N(\Lambda) & := \{\vnu \in \mc F\setminus \Lambda:\, \vnu - \ve_m \in \Lambda\; \forall m \in\supp(\vnu) \text{ and }\\
	& \qquad\qquad \nu_m = 0 \; \text{ for } m > \max(\supp(\Lambda)) + m_\text{buffer}\}.
\end{align*}
\end{enumerate}
The set of admissible neighbors $\mc N(\Lambda)$ of $\Lambda$ is defined such that adding any $\vnu \in \mc N(\Lambda)$ to $\Lambda$ maintains monotonicity.
The restriction in the definition of $\mc N(\Lambda)$ above is that we do not allow  the \emph{activation} of any dimension $m\in\bbN$, i.e., including $\vnu = \ve_m$ for arbitrarily (large) $m\in\bbN$, but restrict the selection to the ``next'' $m_\text{buffer}$ higher dimensions.
Moreover, for our numerical simulations, we have chosen
\[
	\hat c_\vnu := \prod_{m\geq 1} (\nu_m)^{2\theta+2-r} \tau_m^{-2(1\wedge \nu_m)}
\]
with $\tau_m = m^{q-1}$, $\theta=1$ and a suitable value\footnote{We used $r = 2\, ( 2(\theta+1) + 2/p + 1) = 10+4(q-1)$ in the numerical simulations.} for $r > 2(\theta+1) + \frac 2p$, cf. the proof of Lemma \ref{lem:c_nu_b_nu}.

\paragraph{A-posteriori algorithm}
Beside this \emph{a-priori} construction which is usually cheap to run, we also apply a more costly a-posteriori algorithm for generating monotone multi-index sets $\Lambda_N$. 
Such an algorithm already appeared in \cite{SchillingsSchwab2013, ChkifaEtAl2014, ChkifaEtAl2015, Chen2016, NobileEtAl2016} and is motivated by using \emph{a-posteriori} heuristics for estimating the improvement of including $\Delta_\vnu u$ in the sparse collocation approximation.
In particular, the a-posteriori algorithm works exactly as the a-priori algorithm except for substituting the choice \eqref{equ:index_choice_prior} by
\begin{align} \label{equ:index_choice_post}
	\vnu_{n}^* := \argmax_{\vnu \in \mc N(\tilde \Lambda_{n-1})} \frac{\left\| \Delta_\vnu u \right\|_{L^\infty_\mu}}{|\Xi^{(\vnu)}|}.
\end{align}
The ratio $\| \Delta_\vnu u \|_{L^\infty_\mu}/|\Xi^{(\vnu)}|$ represents the \emph{profitability} or \emph{profit} of the multi-index $\vnu \in \mc F$, i.e., the associated gain in approximation $\| \Delta_\vnu u \|_{L^\infty_\mu}$ in relation to the associated computational cost $|\Xi^{(\vnu)}|$.
By choosing the most profitable multi-index in the neighborhood of $\tilde \Lambda_{n-1}$ we may obtain  a better sparse collocation approximation than when applying the a-priori construction \eqref{equ:index_choice_prior}, although the theory developed above does not apply to the multi-indices generated in this way.
Here, we estimated $\| \Delta_\vnu u \|_{L^\infty_\mu}$ as in \cite{NobileEtAl2016} by
\[
	\| \Delta_\vnu u \|_{L^\infty_\mu} 
	\approx 
	\max_{\vxi_\vk \in \Xi^{(\vnu)}} \|\rho(\vxi_\vk) \, \left[\Delta_\vnu u\right](\vxi_\vk)\|_{H_0^1(D)},
\]
where $\rho(\vxi) = \exp(- \frac 12 \sum_{m\geq1}\xi_m^2 )$ represents the (unnormalized) product density funtion of $\mu = \bigotimes_{m\geq1} N(0,1)$.
Thus, for the a-posteriori algorithm we have to evaluate $u$ on a much larger grid than just $\Xi_{\tilde \Lambda_N}$, namely,
$
\Xi_{\tilde \Lambda_{N}} \cup \Xi_{\mc N(\tilde \Lambda_{N-1})},
\, 
\Xi_{\mc N(\tilde \Lambda_{N-1})} := \bigcup_{\vnu \in \mc N(\tilde \Lambda_{N-1})} \Xi^{(\vnu)}.
$
We will refer to $\Xi_{\tilde \Lambda_N}$ as the a-posteriori grid (associated with $\Lambda_N$) and to $\Xi_{\tilde \Lambda_{N}} \cup \Xi_{\mc N(\tilde \Lambda_{N-1})}$ as the extended grid (associated with $\Lambda_N$).
The latter represents the ``true'' computational cost of the sparse collocation approximation generated by the a-posteriori algorithm.

\begin{rem}
For our numerical simulations we choose a maximal number of parameter dimensions $M$, which may be arbitrarily large, to construct the refererence solution.
Then, for a given $\vxi \in \bbR^M$ we approximate the solution $u(x,\vxi)$ to \eqref{equ:BVP_1D} by evaluating its exact representation
\[
	u(x,\vxi) = \int_0^x \frac{K(\vxi) - F(y)}{a(y,\vxi)} \, \d y,
	\qquad
	 F(x) := \int_0^x f(y) \, \d y,
	\quad
	K(\vxi) := \frac{\int_0^1 \frac{F(y)}{a(y,\vxi)} \, \d y}{\int_0^1 \frac{1}{a(y,\vxi)} \, \d y},
\]
by numerical quadrature, particularly the trapezoidal rule based on an equidistant spatial grid with spacing $\Delta x = 2^{-10}$.
\end{rem}

%
%
\subsection{Results}
We are now ready to discuss the details and the results of the numerical tests we performed. 
The tests are divided into two parts: in the first set of experiments 
we aim at validating the sharpness of our analysis, 
i.e., whether we can actually observe numerically the rate predicted by Theorem~\ref{cor:conv-final-for-PDE} for the case of countably many random variables; in the second set of experiments,
we will instead gradually increase the number of random variables 
and see if the observed rate of convergence is actually dimension-independent. 
Concerning the first set of experiments, we recall that the convergence results in Theorem~\ref{cor:conv-final-for-PDE} strictly apply only to the sparse collocation constructed by the a-priori index selection algorithm.
However, we will assess whether the set of indices proposed by the a-posteriori construction, i.e., the a-posteriori grid, achieves the same rate and also examine the convergence rate w.r.t.~number of points in the extended grid.

\paragraph{Tests - Part I}

\begin{table}
  \caption{Statistics for numerical results in Tests - Part I. See equation \eqref{equ:rate_1D} for the theoretical rates.}
  \label{table:conv-stats}
  \centering
  \begin{tabular}[tbp]{l|l|ccc|ccc}
     $q$ 	& \bf var. perc.& \multicolumn{3}{c|}{\bf rate w.r.t. $|\Lambda_N|$} 		& \multicolumn{3}{c}{\bf rate w.r.t. $|\Xi_{\Lambda_N|}$} \\	
     		&				& theory	& a-post. 	& a-priori 						& theory	& a-post. 	& a-priori \\
\hline	 
    1		& 99.91\%   	& N.A.		& 0.5 		& 0.4							& N.A.		& 0.5 		& 0.5	\\ 	
    1.5		& 99.9999\% 	& 0			& 0.8 		& 0.7							& 0			& 0.9 		& 0.8	\\ 	
    2		& 99.9999999\% 	& 0.5		& 1.1 		& 1.0							& 0.25		& 1.2 		& 1.1	\\ 	
    3		& 100\%  		& 1.5		& 1.7 		& 1.7							& 0.75		& 2 		& 2	\\ 	
  \end{tabular}
\end{table}

In this section, we will compare the numerical convergence rate of both the a-priori and the a-posteriori
versions of the proposed algorithm against the theoretical convergence rate for $q=1,1.5,2,3$ to verify
the sharpness of our theoretical analysis. 
For each tested value of $q$, the errors will be computed against a reference solution 
$u_{ref}$ based on the first $640$ random variables which captures more than 99\% of the log-diffusion variability
for every value of $q$ (see Table \ref{table:conv-stats} for the precise value). The error is computed
with a Monte Carlo sampling over $N_{MC}=1000$ random samples:
\begin{align}\label{eq:err-def}
\| u - U_{\Lambda_N} u\|_{L^2_\mu(\bbR^\bbN;H^1_0(D))} 
& \approx \| u_{ref} - U_{\Lambda_N} u \|_{L^2_\mu(\bbR^\bbN;H^1_0(D))} \nonumber \\
& \approx \frac{1}{N_{MC}}\sum_{k=1}^{N_{MC}} \| u_{ref}(\vxi_k) - U_{\Lambda_N} u(\vxi_k) \|_{H^1_0(D)},   \end{align}
where $\vxi_k$ are samples drawn from $\bigotimes_{m=1}^{640}N(0,1)$.
We remark that we have verified that $N_{MC}$ is large enough for our purposes.%
\footnote{i.e., repeating the same analysis with $N_{MC}=5000$ produced identical results.}  

\begin{figure}[tpb]
  \centering
  \includegraphics[width=0.48\textwidth]{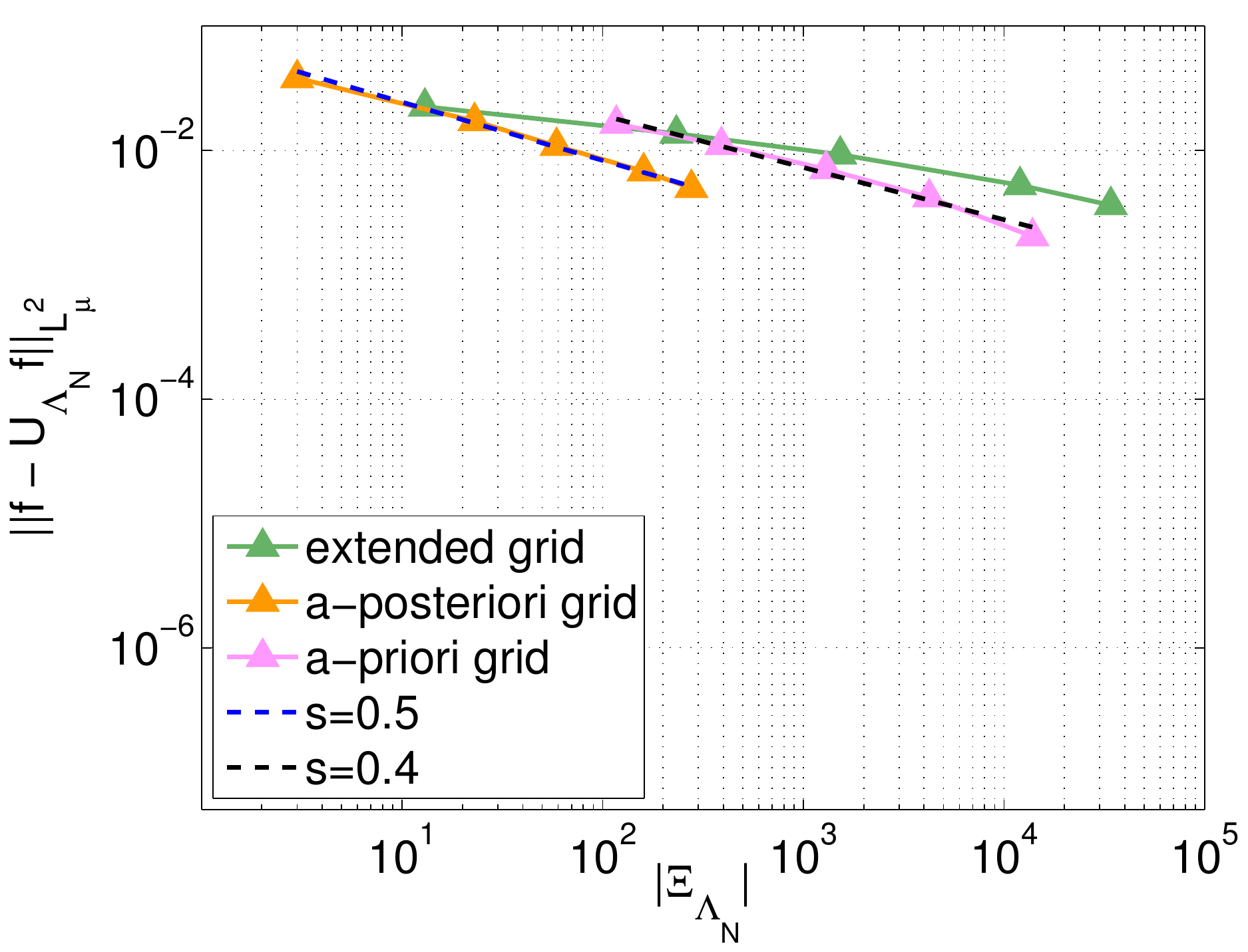}
  \includegraphics[width=0.48\textwidth]{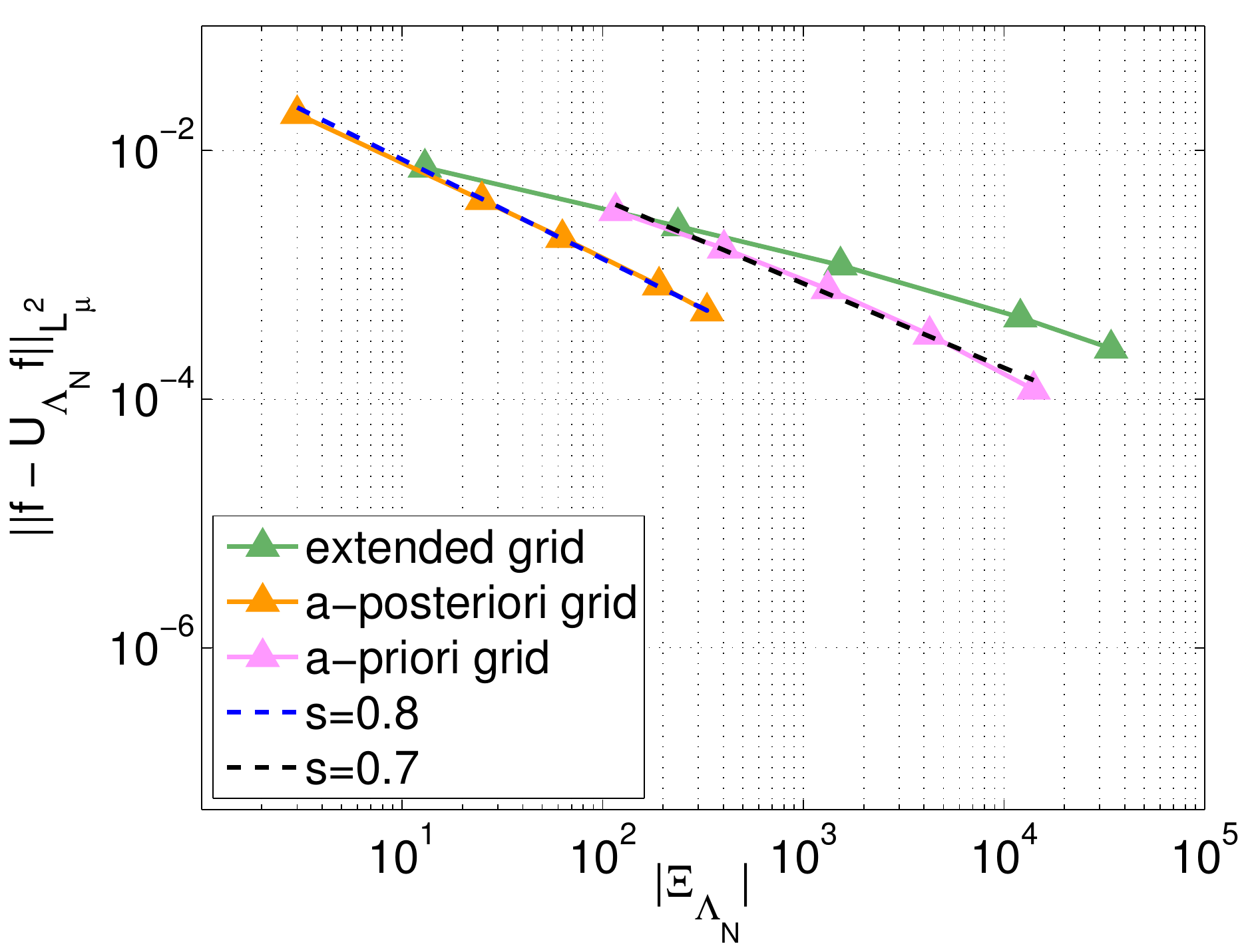} \\
  \includegraphics[width=0.48\textwidth]{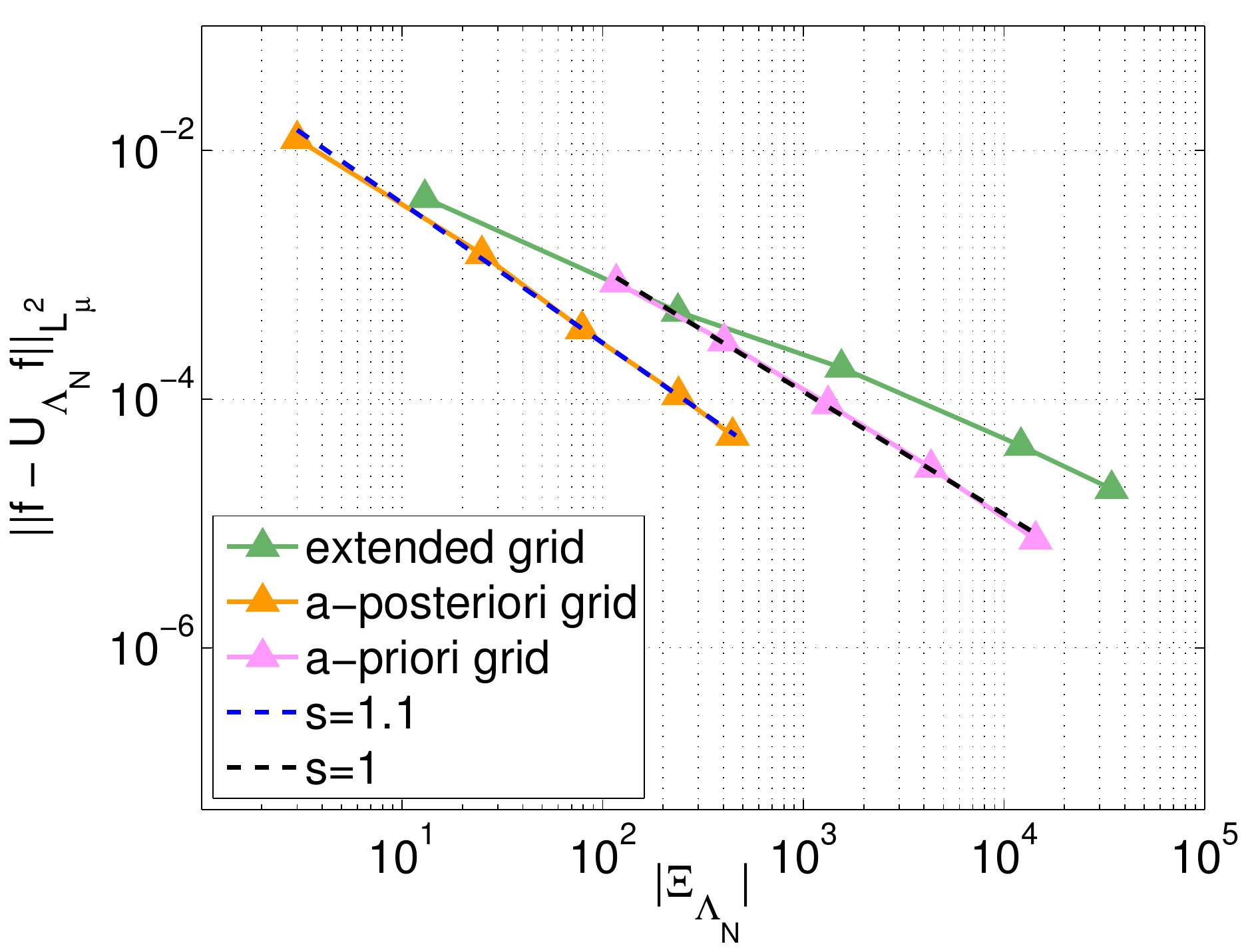}
  \includegraphics[width=0.48\textwidth]{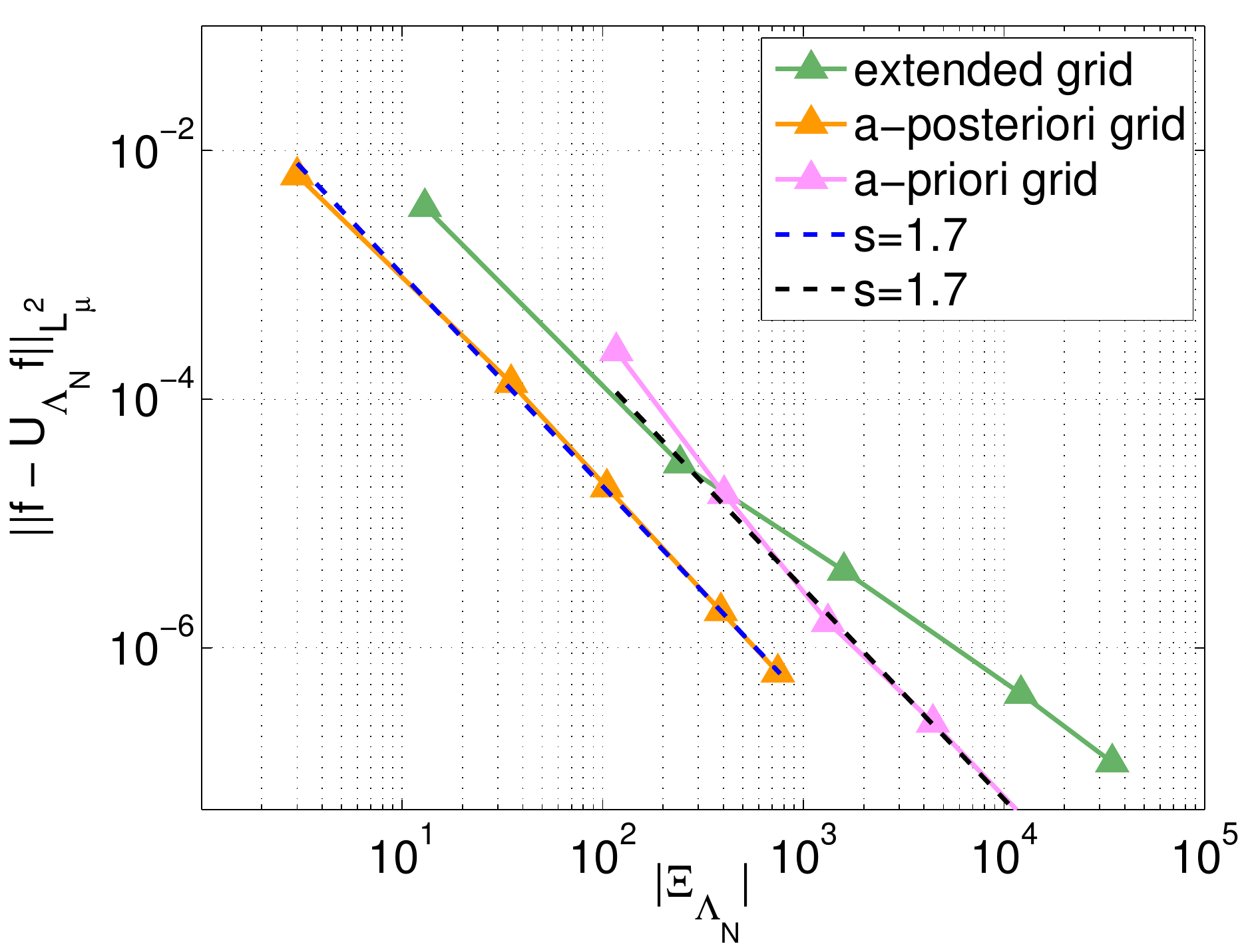}
  \caption{From top-left to bottom-right: convergence with respect to the number of points in the sparse grid for $q=1,1.5,2,3$. }
  \label{fig:err-vs-pts-KL}
\end{figure}

We begin by reporting in Figure \ref{fig:err-vs-pts-KL} the convergence of the error measure \eqref{eq:err-def} with respect to 
the number of collocation points needed to construct the sparse grid approximation for each value of $q$. 
The convergence plots in Figure \ref{fig:err-vs-pts-KL} show a monotone, well-established decreasing trend for the error
for all the variations of the sparse grid considered. 
As expected, the errors get larger in size and the convergence rate gets worse as $q$ decreases 
for all the reported sparse grids (a-posteriori grid, extended grid, a-priori).
In particular, the convergence rate appears to be similar for the a-priori and the a-posteriori algorithm,
with the rate of the latter being actually slightly larger, thus validating the a-posteriori construction. 
On top of this, the error of the a-posteriori algorithm appears to be smaller in size than the a-priori construction. 
We also remark that the rate that we measure numerically is better than the one predicted by our theory, cf. Table \ref{table:conv-stats}.
The quite significant difference between the rate of convergence of the a-posteriori grid
and the extended grid is also to be expected.
These results are consistent with the ones detailed in \cite{Chen2016}, 
although there the a-priori construction is a bit different from the one we propose.
At this junction, two factors can explain the suboptimality of our theoretical result:
a conservative estimate of the growth of the number of points in the sparse grid with
respect to the number of indices in the set $\Lambda_N$ and a conservative link
between the summability of the log-diffusion field representation and the convergence
of the sparse grid. As will be clearer later, both issues turn out to be actually affecting our analysis.

The numbers in the plot show the number of activated random variables in the a-posteriori grid and in the a-priori grid, i.e., 
in how many random variables these grids allocate at least one non-trivial point (observe that by construction
the numbers for the extended grid are the ones of a-posteriori grid plus the buffer $m_\text{buffer}$). 
It can be seen that this number steadily increases. 
The numerical results we show were obtained by $m_\text{buffer}=5$. \footnote{We report (not shown)
that we have also run the same simulations with a larger buffer $m_\text{buffer}=20$ 
and the results were identical (i.e., same a-posteriori grid and same number of activated random variables).}

\begin{figure}[tpb]
  \centering
  \includegraphics[width=0.48\textwidth]{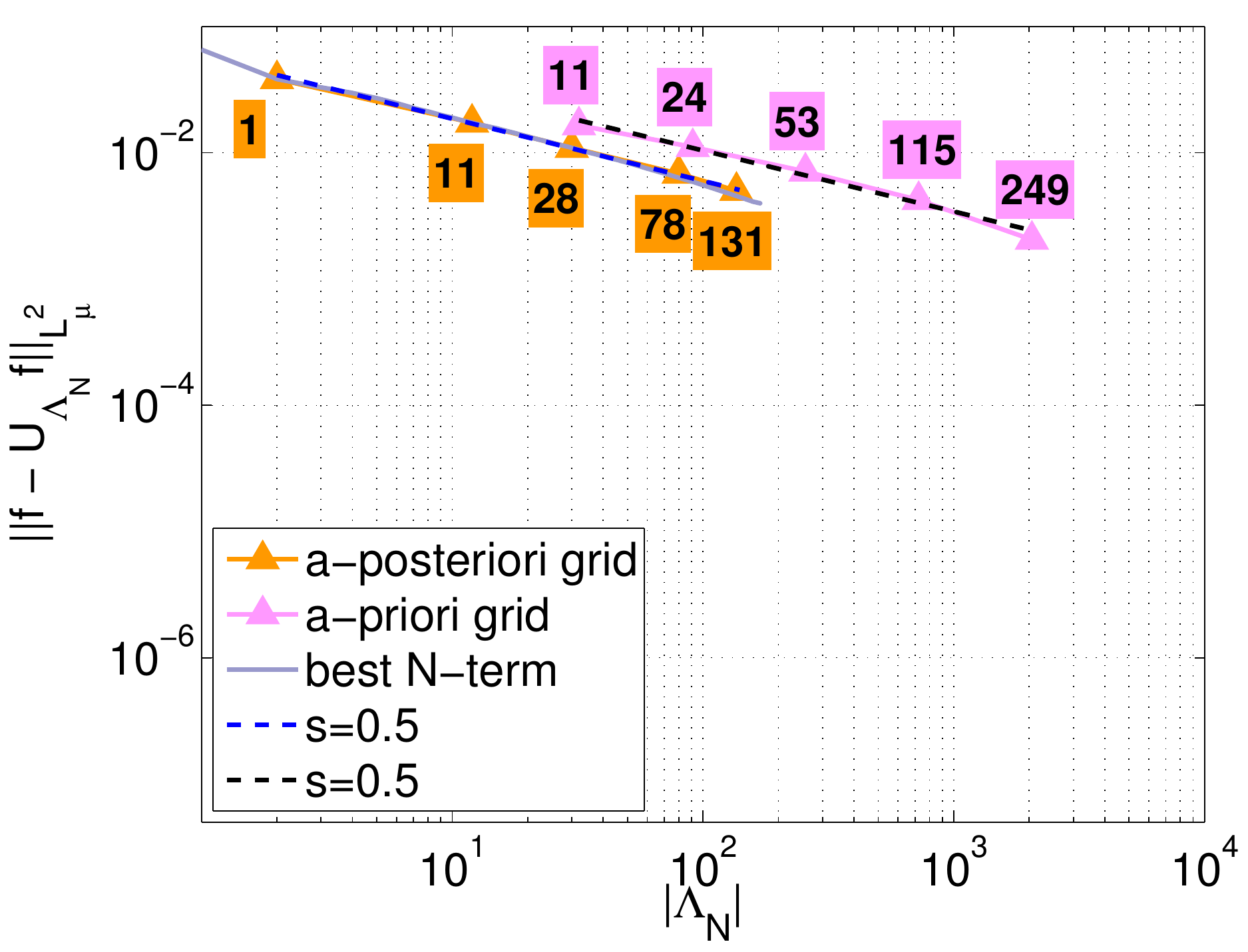}
  \includegraphics[width=0.48\textwidth]{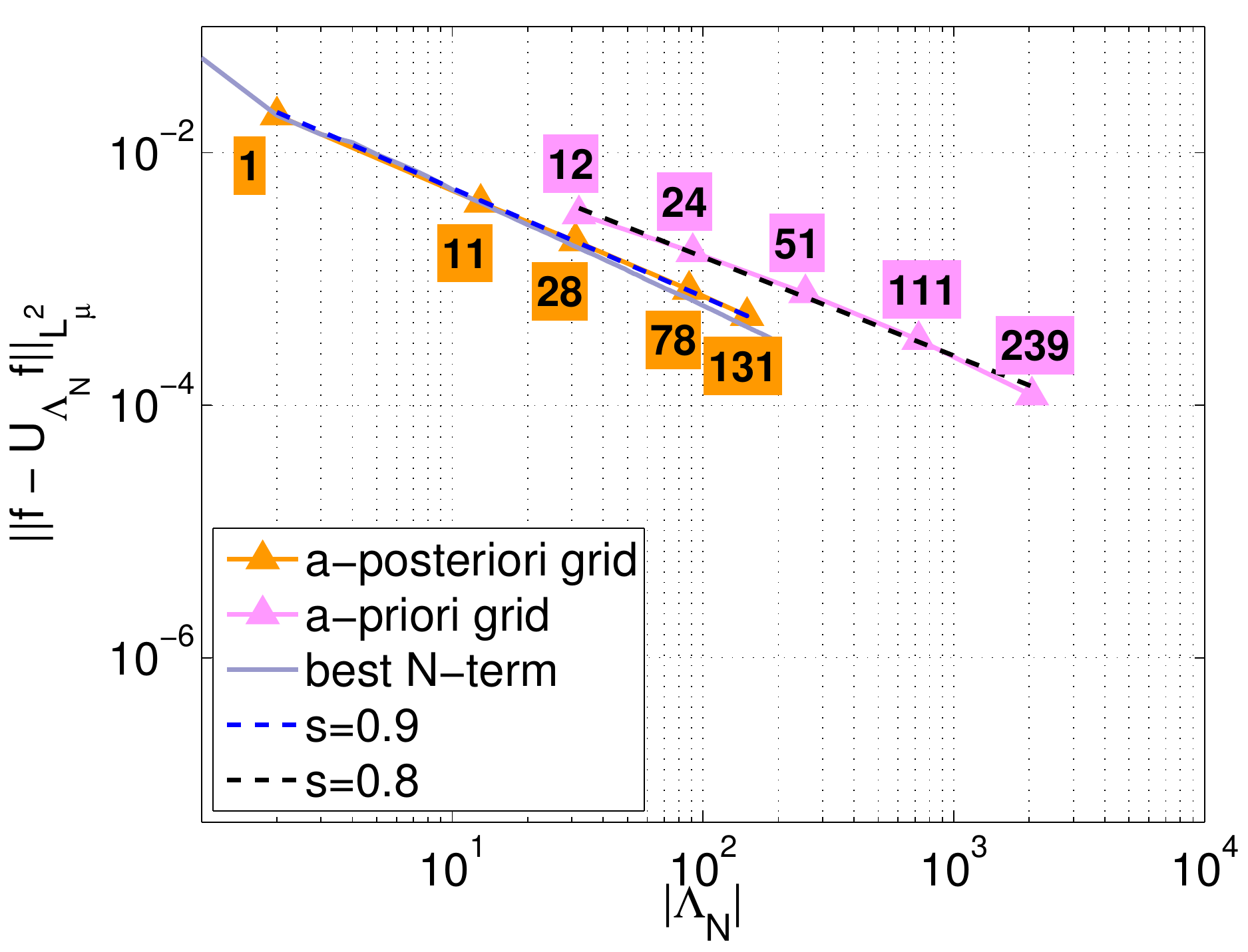} \\
  \includegraphics[width=0.48\textwidth]{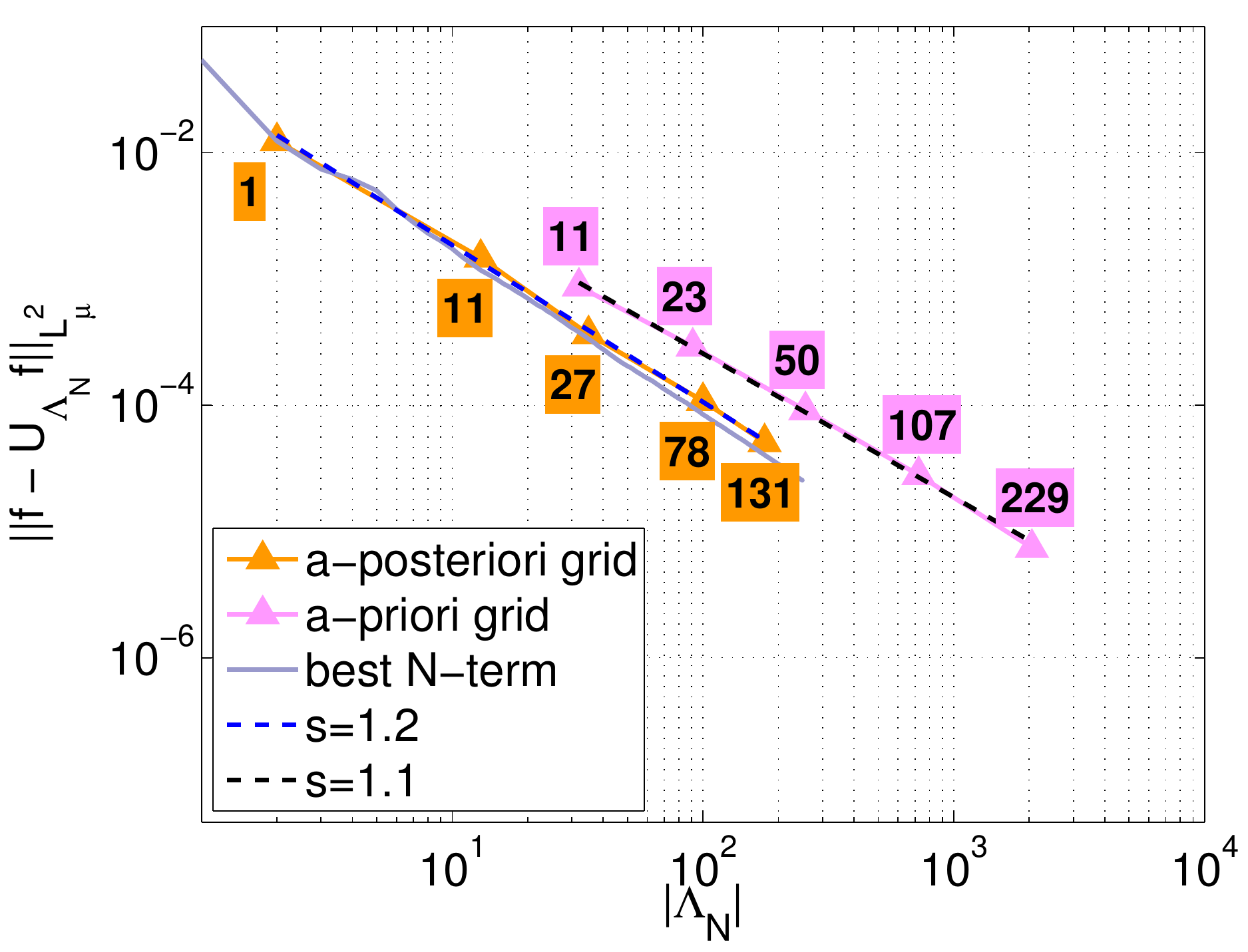}
  \includegraphics[width=0.48\textwidth]{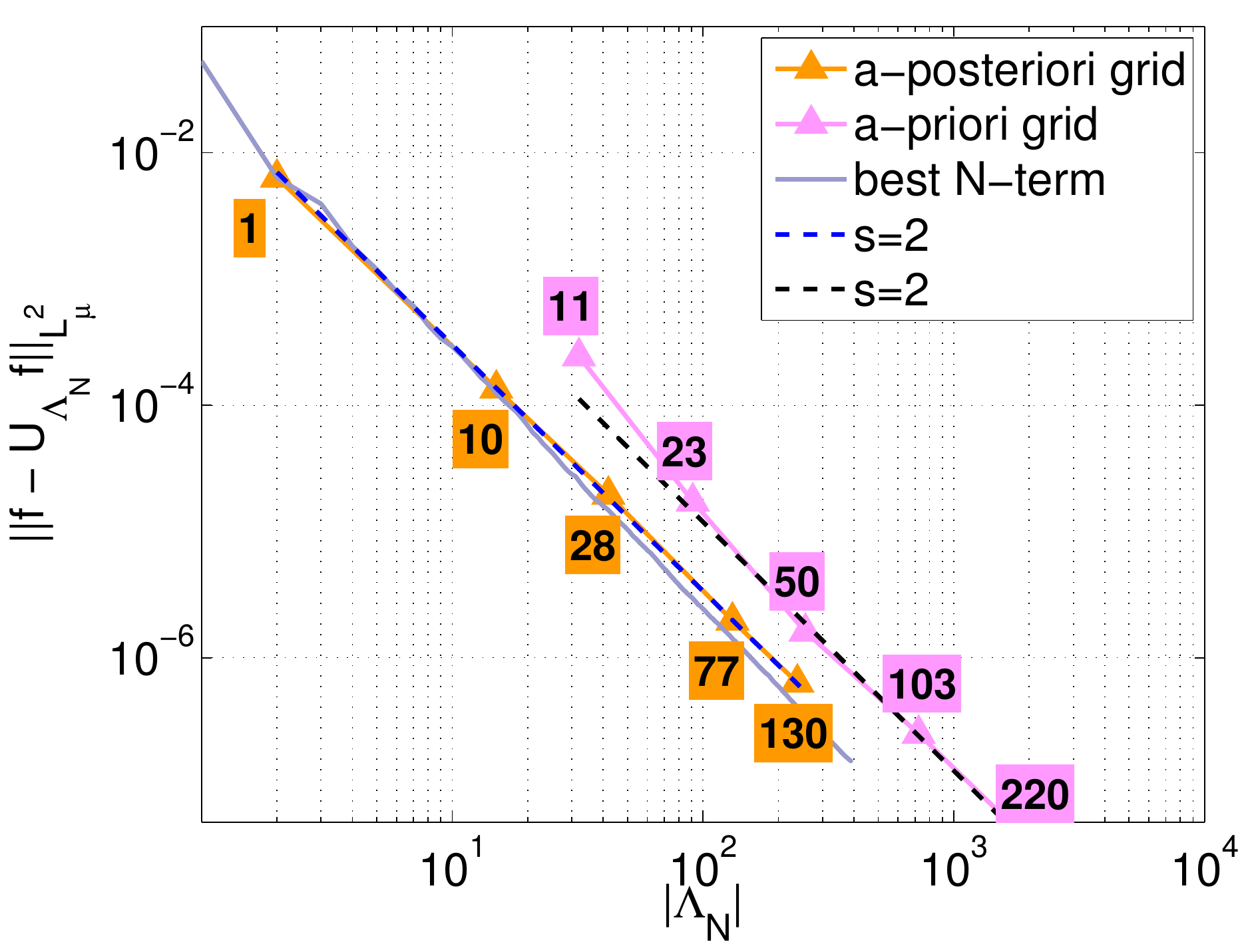}
  \caption{From top-left to bottom-right: convergence with respect to the number of indices in the set $\Lambda_N$ for $q=1,1.5,2,3$. }
  \label{fig:err-vs-idx-KL}
\end{figure}

We then report in Figure \ref{fig:err-vs-idx-KL} the convergence of the error \eqref{eq:err-def} with respect to  the number of indices in the set $\Lambda_N$. 
In this Figure, we show the convergence of both the a-priori and the a-posteriori algorithm, as well as an estimate of the convergence of the best $N$-term  
 approximation of $u$ (we will detail in a moment how we computed this approximation). 
Also in this case, the convergence plots show a monotone, well-established decreasing trend for the error.
The results are similar to the previous case: 
a) the convergence rate of the sparse grid gets worse as $q$ decreases;
b) the convergence rate seems to be identical for both the a-priori and the a-posteriori constructions,
and again quite larger than the theoretical estimate, cf. Table \ref{table:conv-stats};
c) the error of the a-posteriori algorithm is substantially smaller than the one of the a-priori algorithm. 
It is also relevant to notice that the measured convergence rate here is essentially 
identical to the one observed with respect to the number of sparse grid points.
This is in agreement with the results in \cite{Chen2016} and implies that 
for the sparse grids constructed here the growth of number of points w.r.t.~the number of indices is essentially linear, and therefore our
Lemma \ref{propo:X_Lambda} is quite conservative. 

 We now turn the attention to the best $N$-term  approximation presented in Figure \ref{fig:err-vs-idx-KL}. 
To compute this approximation, we follow \cite{feal:compgeo,lever.eal:inversion,Tamellini2012} 
and convert the extended grid first into its \emph{combination technique} form, i.e., as a linear combination of Lagrange polynomials,
and then we further convert this expression into the equivalent linear combination of
Hermite polynomials; see also \cite{constant.eldred.phip:pseudospec}. 
By sorting in decreasing order the coefficients of the Hermite expansion
thus computed and picking them one at a time, we obtain an approximation of the 
sequence of best $N$-term  approximations. \footnote{Of course, this approximation is 
as good as the original extended grid; however, we found the results to be stable as the number of points in the extended grid grows, 
and therefore we deemed this approximation to be sufficient for our purposes.}
The comparison of the best $N$-term  and the a-posteriori grid in Figure \ref{fig:err-vs-idx-KL} reveals
that the two approximations are actually very close a-posteriori grid for every value of $q$,
which suggests that the a-posteriori algorithm is producing an excellent approximation.

\paragraph{Tests - Part II}
In this set of experiments, with fix $q=2, \sigma=0.1$, and we consider log-diffusion coefficients
with $M=10,20,40,80,120,160$ random variables. For each $M$, the reference solution uses $M$ random variables as well, contrary to the previous experiment, where the reference solution was based on $640$ random variables. In this way, we aim at assessing the behavior of the 
convergence rate as $M$ increases: indeed, the previous experiment was only verifying
that we get \emph{a} rate for $M \rightarrow \infty$. We report our results in Figure \ref{fig:err-vs-idx-M},
where we show the convergence with respect to the cardinality of the index set $\Lambda_N$.
It is clearly visible that the convergence curves are all superposed at the beginning 
of the convergence and then they depart from each other: the point of departure
is actually the point where all $M$ variables have been activated. The result
seems to suggest that the convergence rate with respect to the cardinality of $\Lambda_N$
for finite $M$ is actually depending on $M$,
and decreases as $M$ increases, until reaching the asymptotic rate for $M \rightarrow \infty$. 

\begin{figure}[tpb]
  \centering
  \includegraphics[width=0.48\textwidth]{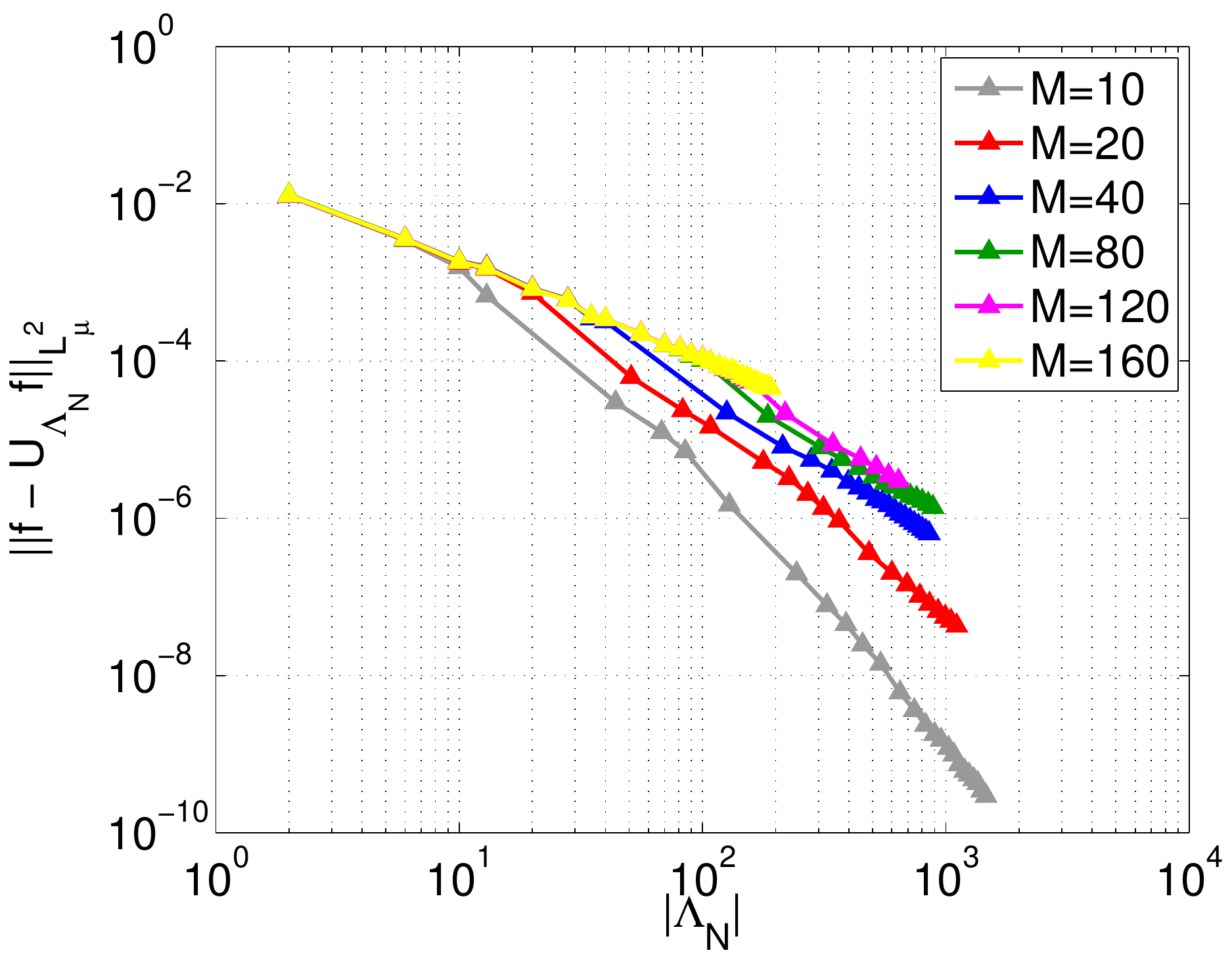}
  \caption{Convergence of the sparse grid approximation with increasingly larger number of dimensions: the asymptotic
    rate is not constant with respect to $M$.}
  \label{fig:err-vs-idx-M}
\end{figure}

%
%
\section{Conclusions}\label{sec:concl}
We have presented a general convergence analysis of sparse grid collocation based on Lagrange interpolation for functions of countably many Gaussian variables.  
In particular, we have stated sufficient conditions on the underlying univariate interpolation nodes such that for functions of a certain smoothness we obtain an algebraic rate of
convergence for the sparse collocation approximation with respect to the number of multi-indices.  
Moreover, we verified these assumptions for the classical Gauss-Hermite nodes and
were able to state also a convergence result in terms of the resulting number of collocation points. 
We finally discussed in detail that these methods can be applied to weak solutions of lognormal diffusion problems and illustrated our theory with numerical tests, which show that the convergence rate achieved by a-priori sparse grid constructions is actually higher than predicted, both with respect to the number of multi-indices and the number of collocation points. 
The classical adaptive a-posteriori sparse grid construction is also seen to achieve such rates, although not covered by our theory.

%
%

\section*{Acknowledgments}
The authors are grateful to Hans-J\"org Starkloff for pointing out the original reference to Stechkin's lemma.

\bibliographystyle{siam}
\bibliography{literature}